\documentclass{article}

\usepackage{geometry}
\usepackage{amsmath,amssymb,amsthm}
\usepackage[numbers]{natbib}
\usepackage[colorlinks,citecolor=blue,urlcolor=black]{hyperref}
\usepackage{graphicx}
	\graphicspath{{images/}}
\usepackage{mathtools}
\mathtoolsset{showonlyrefs}
\numberwithin{equation}{section}
\theoremstyle{plain}
\newtheorem{theorem}{Theorem}[section]
\newtheorem{proposition}[theorem]{Proposition}

\newtheorem{lemma}[theorem]{Lemma}
\theoremstyle{definition}
\newtheorem{definition}[theorem]{Definition}
\newtheorem{example}[theorem]{Example}
\newtheorem{remark}[theorem]{Remark}
\usepackage[utf8]{inputenc}
\usepackage{bbm}
\usepackage{mathtools}
	\mathtoolsset{showonlyrefs}
\usepackage{booktabs}
\usepackage{multirow}
\usepackage{float}

\usepackage{dsfont}
\usepackage{placeins}

\newcommand{\bE}{\mathbb E}
\newcommand{\bN}{\mathbb N}
\newcommand{\bP}{\mathbb P}
\newcommand{\bR}{\mathbb R}
\newcommand{\mean}[1]{\mathbf{m}(#1)}
\newcommand{\sige}[1]{\boldsymbol\sigma_{\mathbf e}(#1)}
\newcommand{\sigt}[1]{\boldsymbol\sigma_{\mathbf t}(#1)}
\DeclareMathOperator{\Var}{Var}
\DeclareMathOperator{\cov}{cov}

\usepackage{accents}
\newcommand{\widebar}[1]{\accentset{\rule{.7em}{0.5pt}}{#1}}
\allowdisplaybreaks

\title{Parameter estimation of integrated\\ fractional Brownian motion}

\author{Marco Mastrogiovanni \footnote{Department of Information Engineering, Computer Science and Mathematics, University of L'Aquila, \url{marco.mastrogiovanni@graduate.univaq.it}} \and Yuliya Mishura \footnote{Department of Probability, Statistics and Actuarial Mathematics, Taras Shevchenko National University of Kyiv, \url{yuliyamishura@knu.ua}}  \and Stefania Ottaviano\footnote{Department of Mathematics ``Tullio Levi Civita”, University of Padova, \url{stefania.ottaviano@math.unipd.it}} \and Kostiantyn Ralchenko \footnote{Department of Probability, Statistics and Actuarial Mathematics, Taras Shevchenko National University of Kyiv, \url{kostiantynralchenko@knu.ua}, and School of Technology and Innovations, University of Vaasa, \url{kostiantyn.ralchenko@uwasa.fi}} \and Tiziano Vargiolu\footnote{Department of Mathematics ``Tullio Levi Civita”, University of Padova, \url{vargiolu@math.unipd.it}} }

\date{}

\begin{document}
\maketitle
\begin{abstract}
Fractional Brownian motion (fBm) is a canonical model for long-memory phenomena. In the presence of large amounts of potentially memory-bearing data, the data are often averaged, which can change the structure of the underlying relationships and affect standard estimation procedures.  To address this, we introduce the normalized integrated fractional Brownian motion (nifBm), defined as the average of fBm over a fixed interval. We derive its covariance structure, investigate the stationarity and self-similarity, and extend the framework to linear combinations of independent nifBms and models with deterministic drift.

For such linear combinations, we establish stationarity of increments, investigate the asymptotic behavior of the autocovariance function, and prove an ergodic theorem essential for statistical inference. We consider two statistical models: one driven by a single nifBm and another by a linear combination of two independent nifBms, including cases with deterministic drift. For both models, we propose estimators that are strongly consistent and asymptotically normal for both the drift and the full parameter set.

Numerical simulations illustrate the theoretical findings, providing a foundation for modeling averaged fractional dynamics, with potential applications in finance, energy markets, and environmental studies.
\\[11pt]
\textbf{2020 Mathematics Subject Classification.}
Primary 60G22; secondary 60G15, 62F12.
\\[8pt]
\textbf{Key words and phrases.}
Fractional Brownian motion,
integrated processes,
Gaussian processes,
mixed fractional models,
parameter estimation,
strong consistency,
asymptotic normality.
\end{abstract}

\section{Introduction}

In the past decade, models involving fractional dynamics have become widely used across various scientific fields, including communications, biology, physics, and finance. The growth of interest around these models is primarily due to the ability of fractional dynamics to describe systems with memory effects, where future values are influenced by past values rather than being entirely independent. The most notable example of such processes is fractional Brownian motion (fBm), which is characterized by its Hurst exponent $H \in (0,1)$: besides characterizing the smoothness of fBm paths, the Hurst exponent is a measure of its long-term memory, capturing the degree of autocorrelation and  its decaying over time. 

As it is very well-known,  the Hurst exponent  $H$  provides insight into the nature of a time series based on a fBm, allowing classification into three categories: if $H=0.5$ the time series is a random walk; if $H<0.5$ indicates anti-persistence which means that an up value is more likely to be followed by a down value, and vice versa; if $H>0.5$ the series is said to be persistent, it means that the direction of the next value is more likely to be the same as the current value.
 
There are numerous methods to estimate the Hurst exponent and characterize short- or long-range dependence, based on the behavior of the autocovariance function (ACVF) and on the self-similarity properties. A key issue in the literature, which we aim to highlight, is that in many studies researchers pre-process the data, often averaging it, before estimating $H$. This pre-processing step may introduce distortions that affect the accuracy of the exponent's estimation: in fact, nothing guarantees that a linear combination of increments of a fBm, like the typical result of this averaging process, is still a fBm. Common examples of this phenomenon are time series of spot electricity prices in markets having an hourly frequency: in fact, a quite common approach in literature is to calculate the series using a daily average before performing the estimation of $H$.

The main motivation of this paper is to try and correct this approach, by studying the process that we call normalized integrated fractional Brownian motion (nifBm), defined as 
\[
X_t^{h,H}:=\frac{1}{h} \int_{t}^{t+h} W^H_u \;du.
\]
where $h > 0$ is fixed and $t$ varies in the interval $[0,T]$, with $T \in (0,+\infty)$, and some processes related to it. This process is, of course, Gaussian, thus its distribution is fully characterized by its covariance operator, which in turn will depend on the increment $h$ and on the Hurst exponent $H$. Extensions of this process that we will consider in this paper are linear combinations of two independent nifBm, i.e.\ $a X_t^{h,H_1} + b X_t^{h,H_2}$, with $a, b \neq 0$ and $H_1 \neq H_2$, and drifted versions of these processes, where the drift is of the form $\mu G(t)$, with $G\colon[0,T] \to \bR$ known (possibly nonlinear) function and $\mu \neq 0$ being a parameter. We feel that this new class of processes could be easily applied to the modeling of phenomena where one does not observe directly something happening instantaneously (which could have the dynamics represented as a fBm), but rather the continuous accumulation of it in a given time interval of length $h$ (which could be represented as the increment of a nifBm). For example, two common ways to measure insolation in a given location are the solar radiance, which is an instantaneous measurement of solar radiation, and solar insolation, which is the solar radiance averaged over a given time period \cite[Chapter 2.5]{HonBow}. Thus, if one assumes that solar radiance follows a fBm, as a logical consequence, solar radiation should follow a nifBm. 

Without claiming to give a complete presentation of all works devoted to integrated fBm, let us mention nevertheless that it is widely studied both from a theoretical point of view and for applications. For example, in \cite{LiShao}  small ball estimates for integrated fBm are considered, in \cite{MR2035633} such  functionals of integrated fBm as the maximum, the position of the maximum, the occupation time above zero and some others are studied. In \cite{Molchan2018} the problem of the log-asymptotic behavior of the probability that the integrated fBm does not exceed a fixed level over a long time is solved. The authors of \cite{MR4466414} consider the persistence exponent of integrated fBm, a problem that originates in theoretical physics, where the persistence exponent serves as a simple measure of how fast a complicated physical system returns from a disordered initial condition to its stationary state. In \cite{fractalfract7020107} the persistence exponent is studied numerically, while in \cite{abundo2019integral} the mean and covariance of ifBm and of some other integrated Gaussian processes are investigated. In \cite{MR4880510} biomarker measurement modeling with a longitudinal submodel is treated, and to overcome longitudinal submodel challenges, the authors replace random slopes with scaled integrated fBm.
  
Linear combinations of two fractional processes have also been the subject of extensive research. The motivation for this approach is that one process typically captures short-term variations and exhibits rougher behavior, whereas the other represents long-term dynamics and is smoother. However, most existing studies concentrate on linear combinations of two fractional Brownian motions (fBms) or, as a special case, on the so-called  mixed fBm,  which is defined as a linear combination of a Wiener process and an fBm. However, we name mixed fBm a linear combination of any fBms. For instance, the linear combination of two fBms was examined in \cite{Ralchenko23}, where estimators were proposed for the coefficients $a$ and $b$, as well as for the two Hurst exponents $H_1$ and $H_2$. The problem of maximum likelihood estimation of the drift parameter in a model with linear drift driven by two independent fBms was studied in \cite{Mishura2016,MRZ24}. Moreover, \cite{Mishura2018NonlinearDriftMLE, Mishura2018twoest} introduced estimators for the parameter $\mu$ in a nonlinear drift term $\mu G(t)$, where the driving process is any Gaussian process with a known and nonsingular covariance operator.

The particular case of mixed fBm with one Wiener component has been investigated in greater detail from theoretical, practical, and statistical perspectives. Several methods for parameter estimation in such processes have been proposed, including the maximum likelihood method \cite{XZZ2011}, Powell’s optimization algorithm \cite{ZSX2014}, and approaches based on power variations \cite{DMS2015}. Simultaneous parameter estimation for mixed fBm with drift has also been addressed using different methodologies: a modification of maximum likelihood approach in \cite{DPS2020}, and two further approaches -- power variation–based estimators and ergodic-type estimators -- in \cite{Kukush2021}. 

Combining these two important points, namely,   averaging of data with memory and statistical estimation in mixed fractional models (linear combinations of fBms), we come to a new and logical continuation: statistical estimation in averaged (integrated and normalized) fractional models and their linear combinations.

The paper is organized as follows.
In Section~\ref{sec2}, we compute the covariance function of nifBm and of its increments, and investigate the properties of stationarity and self-similarity.
Section~\ref{sec3} is devoted to the study of linear combinations of two nifBms. In particular, we establish the stationarity of increments, analyze the asymptotic behavior of the corresponding autocovariance function, and derive an ergodic theorem, which plays a key role in parameter estimation.
The subsequent sections address statistical problems related to the processes introduced above. Specifically, Section~\ref{sec4} presents a statistical model involving two nifBms with drift, examines two estimators of the drift parameter, and establishes their strong consistency and normality. As a special case, we also discuss the corresponding model with a single nifBm.
The simultaneous estimation of all parameters of a nifBm and of a linear combination of two nifBms is studied in Sections~\ref{constr}--\ref{sec:asnorm}. Section~\ref{constr} focuses on the construction of strongly consistent estimators, while Section~\ref{sec:asnorm} establishes their joint asymptotic normality.
Section~\ref{sec:simul} presents numerical simulations that support our theoretical results.
Finally, Appendix~\ref{app:A} contains proofs of several auxiliary lemmas, and Appendix~\ref{app:gaus-vect} summarizes useful properties of Gaussian random vectors together with a limit theorem employed in the proof of asymptotic normality.

\section{Normalized integrated fractional Brownian motion}\label{sec2}

In this section, we aim to compute covariances and investigate the properties of stationarity and self-similarity of the normalized integral over time of fractional Brownian motion  and of its increments.

\subsection{Basics of fractional Brownian motion} Let $(\Omega,\mathcal{F},\bP)$ be a probability space on which all the processes considered below are correctly defined.

We recall that fractional Brownian motion $W^H=\{ W^H_t, t\ge 0\}$ with Hurst index $H\in(0,1)$ is a centered   Gaussian process  having the following covariance function:
\[
\bE \left[W_s^H W_t^H \right]=\frac{1}{2}\left(s^{2H}+t^{2H}-|s-t|^{2H} \right).
\] 

For $H=\frac{1}{2}$, the fBm is  a classical Brownian motion. If $H=1$, then $W_t^H=tW_1^H$ a.s. for all $t \geq 0$.
For $H > 1/2$, the increments of the process are positively
correlated, while for $H < 1/2$ they are negatively correlated. Moreover, this process has
the self-similarity property, with parameter of self-similarity (or scaling exponent) $H$ and posses stationary increments.
Since we will be interested in the incremental covariances, we recall here the result that can be  obtained by direct calculations and is well-known.  
\begin{lemma}\label{lemincr} For any $H\in(0,1)$ and any $ 0\le s\le t\le u\le v$
$$\mathbb{E}(W_t^H-W_s^H)(W_v^H-W_u^H)=\frac12\big((v-s)^{2H}+(u-t)^{2H}-(v-t)^{2H}-(u-s)^{2H}\big).$$    
\end{lemma}
 Recall also that, according to \cite{Norros1999}, fBm $W^H$ is a Gaussian-Volterra   process admitting a representation
\begin{equation} \label{WH}
W^H_t := \int_0^t K_H(t,s)\ dW_s \qquad \mbox{ for all } \qquad t \geq 0,
\end{equation} 
 where $W$ is a Wiener process (so called underlying Wiener process),   kernel function $K_H$ has a form
$$ K_H(t,s) = c_H \left[ \left(\frac{t}{s}\right)^{H-\frac{1}{2}} (t-s)^{H-\frac{1}{2}} +\left(\frac{1}{2}-H \right) \int_s^t u^{H-\frac{3}{2}}\left(\frac{u}{s}-1\right)^{H-\frac{1}{2}} du\right], $$ 
  for  $H > \frac12$ it is simplified to 
 
$$ K_H(t,s) = \left(H-\frac12\right)c_H s^{1/2-H}\int_s^t {u}^{H-\frac{1}{2}}(u-s)^{H-\frac{3}{2}}du, $$ 
 and $c_H:=\left(\frac{2 H \Gamma(H+\frac{1}{2})\Gamma(\frac{3}{2}-H))}{\Gamma (2-2H)} \right)^{1/2}$. It is also well known, see e.g. \cite{mishura2008stochastic}, that the filtrations generated by $W^H$ and $W$ coincide. Denote this common filtration as $\mathcal{F}=\{\mathcal{F}_t,\,t\ge 0\}.$

\subsection{Covariances and incremental covariances of the normalized integrated fractional Brownian motion. Stationarity of increments}  Now, taking into account that the trajectories of fBm are continuous a.s., for any $ t\ge 0$ we can consider the normalized Riemann integral of the form 
\begin{equation}\label{Xintmean}
X_t^{h,H}:=\frac{1}{h} \int_{t}^{t+h} W^H_u \;du. 
\end{equation}
 In the rest of the paper, we refer to this process as the normalized integrated fractional Brownian motion (nifBm). 
 If the parameter $H$ is clear from the context, we will write simply $X_t^h$. Of course, under any consideration, even being considered as a function of both parameters, $t$ and $h$,  nifBm is Gaussian and zero-mean process. Therefore, its main characteristics is a covariance function.  
So, let us calculate covariance and variance of this process and of its increments. 
In these calculations the length $h$ of the interval of integration will be fixed while $t$ will be shifted. 
Note that some related calculations for variance and covariance of the integrated fBm were obtained  in \cite{abundo2019integral}. However, in order to get the statistical estimators, we need to produce  much more detailed analogs. 
 
First, let us compute the covariance of two values of nifBm. For the readers' convenience, we postpone  the proof of the following lemma   to  the Appendix.

\begin{lemma}\label{covarfar} For any $s \geq t$
\begin{equation}\begin{gathered} \label{covfar}
     \bE[X_t^h X_s^h]= \frac{1}{2h(2H+1)}\big({(s+h)}^{2H+1}-s^{2H+1}+ (t+h)^{2H+1}-t^{2H+1}\big)\\+\frac{1}{2h^2(2H+1)(2H+2)}\big(2(s-t)^{2H+2}-(s-t+h)^{2H+2}-|s-t-h|^{2H+2}\big).  \end{gathered}
\end{equation}   
 In particular, for any $t\ge 0$
 \begin{equation}\label{var1}
\bE[(X_t^h)^2] = \frac{1}{h(2H+1)}\big((t+h)^{2H+1}-t^{2H+1}\big)-\frac{h^{2H}}{(2H+1)(2H+2)}.  \end{equation} 
\end{lemma}

\begin{remark} \label{rem:varX}
It is clear from the formula  \eqref{covfar}, more precisely, from the first term in its right-hand side, as well as  from \eqref{var1}, that the process $X^h$ is not stationary in argument $t$. If we analyze the value  $\bE[(X_t^h)^2]$ from \eqref{var1} as $t\to \infty$, then   $$(t+h)^{2H+1}-t^{2H+1}=t^{2H+1}\left(\left(1+\frac ht\right)^{2H+1}-1\right)\sim t^{2H+1}(2H+1)\frac ht, $$
whence  $\bE[(X_t^h)^2]\sim t^{2H}$ as $t\to\infty$, the same asymptotics as the fBm has.
\end{remark}

Now our goal, taking into account the subsequent construction of statistical estimators that will be based on the ergodic theorem, is to construct a stationary process, starting from  nifBm. In this regard, let us consider its increments obtained by shifting time $t$ by an amount $nh$ proportional to $h$. It is clear that the form of covariance will be different for $n=0$ (variance, intervals of integration are  completely overlapping), $n=1$ (covariance, intervals are partially overlapping) and $n\ge 2$ (covariance, intervals are not overlapping). Taking this into account, introduce the notation \begin{equation*}\begin{gathered}\gamma(H,n) := \begin{cases}
\displaystyle \frac{2(2^{2H}  -1)}{ (2H  +1)( H +1)}, \, n = 0,\\     
\displaystyle ( 2 (2H +1)(2H +2))^{-1}( 7-4 \cdot2^{2H +2}+3^{2H +2}),\, n = 1, \\
    \displaystyle  ( 2 (2H +1)(2H +2))^{-1}((n-2)^{2H +2}\!-4(n-1)^{2H +2}\!\\+6n^{2H +2}-4(n+1)^{2H +2}+\!(n+2)^{2H +2}), \, n \geq 2,
    \end{cases}
\end{gathered}\end{equation*}
 that can be expressed in a more concise form as
\begin{equation}\begin{gathered} \label{gammy}
\gamma(H,n)=\frac{(|n-2|^{2H +2}-4|n-1|^{2H +2}+6|n|^{2H +2}-4|n+1|^{2H +2}+ |n+2|^{2H +2})} { 4 (2H +1)( H +1)},\end{gathered}\end{equation}
and proceed with the covariance of  the increments of the process  $X^h$.

\begin{theorem}\label{theoremfour}
  For any $t\ge 0$ and $n\ge 0$
    \begin{equation}\label{covfin}
   \bE[(X_{t+h}^h- X_t^h)(X_{t+(n+1)h}^h- X_{t+nh}^h)] =  h^{2H}\gamma(H,n). 
  \end{equation}
\end{theorem}
 \begin{proof} The proof is completely based on Lemma \ref{covarfar}, and it is just necessary  to calculate, according to the formulas \eqref{covfar} and \eqref{var1}, all the covariances and then add them up. Therefore, we will consider only the case $n=0$ in detail to show how the terms containing $t$, mutually cancel each other. The covariances that correspond to the cases  $n=1$ and $n\ge 2$ are calculated in the same spirit, with the help of \eqref{covfar} and \eqref{var1}, therefore  we shall omit these very similar but rather tedious calculations. So, let $n=0$. Then, according to  \eqref{var1},
\begin{equation*} 
\bE[(X_{t+h}^h)^2] = \frac{1}{h(2H+1)}\big((t+2h)^{2H+1}-(t+h)^{2H+1}\big)-\frac{h^{2H}}{(2H+1)(2H+2)},  \end{equation*} 
$\bE[(X_{t}^h)^2]$ is calculated in \eqref{var1}, and  the value  $\bE[X_t^h X_{t+h}^h]$ corresponds to \eqref{covarfar} with $s=t+h,$ and  equals $$ \bE[X_t^h X_{t+h}^h]= \frac{1}{2h(2H+1)}\big({(t+2h)}^{2H+1} -t^{2H+1}\big) +\frac{1}{2h^2(2H+1)(2H+2)}\big(2h^{2H+2}-(2h)^{2H+2} \big).$$
Therefore 
\begin{align*}
         \bE (X_{t+h}^h- X_t^h)^2 &=\frac{1}{ h(2H+1) }\big[(t+2h)^{2H+1}-(t+h)^{2H+1}\\
         &\qquad\qquad\qquad+(t+h)^{2H+1}-t^{2H+1}-(t+2h)^{2H+1}+t^{2H+1}\big]\\
         &\quad-\frac {2h^{2H}}{(2H+1)(2H+2)}-\frac{2h^{2H+2}-(2h)^{2H+2}}{h^2(2H+1)(2H+2)}\\
         &=   \frac{h^{2H}}{ (2H+1)(2H+2)}\Big(2^{2H+2}-4  \Big)
         =\frac{2h^{2H}}{ (2H+1)( H+1)}\Big(2^{2H}-1 \Big).
\qedhere
\end{align*}
\end{proof}

Now it is interesting to study for which values of $H\in(0,1)$ the quantities $\gamma(H,n)$ are  positive  or vice versa, 
and compare this to the increments of fBm that are positively or negatively correlated for $H>1/2$ or $H<1/2$, respectively. 

\begin{lemma} \label{H0}\begin{itemize}
    \item[$(i)$] Let $H\in(0,1)$. The function $\gamma(H,1)$,\ is zero for $H = 0$ and admits a unique $H_0 \in (0,1)$ such that $\gamma(H_0,1) = 0$, $\gamma(H,1) < 0$ for $ H\in(0,H_0)$ and $\gamma(H,1) > 0$ for $ H\in(H_0,1]$. Numerical computations give  $H_0 \simeq 0.2626229$. The graph of the function $\gamma(H,1)$ over the interval $(0,1)$ is shown in Figure~\ref{fig:grafico}.
     \item[$(ii)$] For any $n\ge 2$ function $\gamma(H,n),\, H\in(0,1)$ is strictly negative for $H\in (0,1/2)$, zero for $H=1/2$ and strictly positive for $H\in(1/2,1)$. 
\end{itemize}
\end{lemma}

\begin{figure}[ht]
    \centering
    \includegraphics[width=0.8\textwidth]{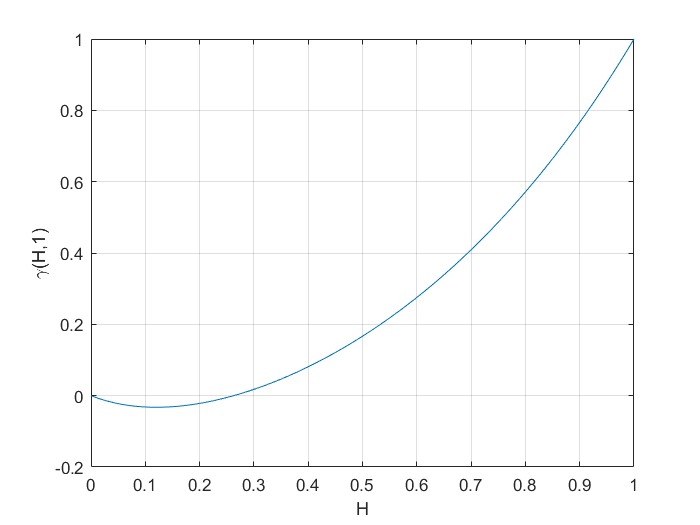} 
    \caption{Graph of the function $\gamma(H,1)$ for $H \in (0,1)$. The function becomes zero at $H_0 \simeq 0.2626229$.}
    \label{fig:grafico}
\end{figure}

\begin{remark}
Note that the covariance in \eqref{covfin} does not depend on $t$. Also, $X_{t}^h$ is a Gaussian process in both arguments $t$ and $h$. It means that choosing $t$ proportional to $h$ (as done in the main part of the paper), we get a stationary Gaussian sequence
$\{X_{(k+1)h}^h - X_{kh}^h, k \in \bN_0\}$.
Also, it is interesting to observe that in the brackets in the right-hand side of  \eqref{covfin} we get the subsequent binomial coefficients of the binomial expansion  for $(a-b)^4.$ This tendency is partially observed even in the numerator of $\gamma(H,1) $ that can be rewritten as $(3^{2H+2}-4\cdot2^{2H+2}+6)+1$, where 1 is forced to replace the remaining terms. 
\end{remark}

\subsection{Self-similarity of the normalized integral  }

Recall  that  fBm is a self-similar process with self-similarity index $H$. This property may be crucial in many estimation procedures. Let us establish self-similarity of  the process $X_t^h,\, t\ge 0.$ 
The standard definition of self-similarity  is the following one. 
\begin{definition} \label{selfsimilarity}
    A stochastic process $\{ Y_t , t \geq 0\}$ is called self-similar with self-similarity parameter $\alpha \in (0,1)$ if for all $c > 0$ 
    $$\{Y_{ct},t\ge 0\}\stackrel{\text{d}}{=} \{c^{\alpha}Y_t, t\ge 0\}.$$
\end{definition}

 Given that nifBm defined in \eqref{Xintmean} depends on $t$ and $h$, we consider the following modified property of self-similarity.
 \begin{lemma} \label{selfInt} For any $h>0$ a nifBm $X^h=\{X_t^h, t\ge 0\}$ is self-similar with self-similarity parameter $H$ in the sense that 
  for all $c > 0$ 
  \begin{equation}\label{ssmod}
      \{X^{ch}_{ct},t\ge 0\}\stackrel{\text{d}}{=} \{c^{H}X^h_t, t\ge0\}.  \end{equation}  
 \end{lemma}

\subsection{Nonsingularity of the incremental covariance matrix}
Fix some $h>0$ and define the discrete-time  increments $\Delta X_k^h :=  X_{(k+1)h}^h - X_{kh}^h$, $k \in \bN_0$. This of course defines a Gaussian sequence 
$\Delta X^h = \{X_{(k+1)h}^h - X_{kh}^h, k \in \bN_0\}$. 
From Theorem \ref{theoremfour}, it is immediate to notice that 
\begin{equation}\label{cov-deltaX}
\bE[ \Delta X_k^h \Delta X_{k+n}^h ] = h^{2H} \gamma(H,n), \qquad \mbox{ for all } k \in \bN_0. 
\end{equation}
 
In the next lemma, we show that any covariance matrix of the increments $\Delta X_k^h$ is non-singular. 
\begin{proposition}\label{prop:covmatrnonsing}
If $0 \leq k_1 < \ldots < k_n$, then $\Sigma_{k_1,\ldots,k_n} := (\mathrm{Cov}(\Delta X^h_{k_i},\Delta X^h_{k_j}))_{i,j = 1 }^n$ has rank $n$.
\end{proposition}
\begin{proof}
Let us prove this by contradiction. Assume that for some $0 \leq k_1 < \ldots < k_n$ $$\mathrm{rank}\ \Sigma_{k_1,\ldots,k_{n-1}} =  \mathrm{rank}\ \Sigma_{k_1,\ldots,k_n} = n - 1$$ (if $\mathrm{rank}\  \Sigma_{k_1,\ldots,k_{n-1}} < n - 1$, a similar argument could be applied). Then it would be possible to express $\Delta X_{k_n}^h$ as a linear combination of the other components, i.e. there would exist $a_1, \ldots, a_{n-1} \in \bR$ such that
\begin{equation}\label{equality1} \Delta X^h_{k_n} = \sum_{i=1}^{n-1} a_i \Delta X^h_{k_i}. \end{equation}
Note that \begin{equation}\label{equality2} \Delta X^h_{k_n} =\frac1h\left(I_n+\zeta_n\right),\end{equation}
where $I_n=\int_{(k_n+1)h}^{(k_n+2)h}W_s^Hds$ and $\zeta_n$ is $\mathcal{F}_{(k_n+1)h}$-adapted. Also, $\sum_{i=1}^{n-1} a_i \Delta X^h_{k_i}$ is $\mathcal{F}_{(k_{n-1}+2)h}\subset\mathcal{F}_{(k_n+1)h}$-adapted.  Taking into account this adaptedness and \eqref{equality1}-\eqref{equality2}, we conclude that $I_n$ is $\mathcal{F}_{(k_{n}+1)h}$-adapted, and therefore $\bE(I_n|\mathcal{F}_{(k_{n}+1)h})=I_n$. However,  according to representation \eqref{WH},
\[
\bE(I_n|\mathcal{F}_{(k_{n}+1)h})= \int_{(k_n+1)h}^{(k_n+2)h}\!\! \int_0^{(k_{n}+1)h}K_H(s,u)dW_uds\neq I_n=\int_{(k_n+1)h}^{(k_n+2)h} \!\!\int_0^{s}K_H(s,u)dW_uds,
\]
whence the proof follows.
\end{proof}
\begin{remark} In fact, we did not use in the proof stationarity of increments $\Delta X_k^h$. However, taking into  account stationarity and assuming that equality \eqref{equality1} holds, we get that for any $m\ge 0$ $$\Delta X^h_{k_n+m} = \sum_{i=1}^{n-1} a_i \Delta X^h_{k_i+m}.$$ Now, by the same reason, we can present all $\Delta X^h_{k_i+m}$ with $k_i+m\ge k_n$ via linear combination of the increments with smaller indices. By induction, it means that assumption \eqref{equality1} leads to the possibility to predict without error all the future increments $\Delta X_k^h$ for $k\ge k_n$ via increments until $k_{n-1}$. This observation is useful in studying stationary Gaussian sequences.    
\end{remark}

\section{Linear combination of two nifBms and its main properties}
\label{sec3}
In this section we study  the properties of linear combination of two independent  incremental processes constructed from two    nifBms.
\subsection{Stationarity of the increments of the linear combination of nifBms}\label{stationar}

Starting with the process \eqref{Xintmean}, put in it $t = kh$ and define the following  stochastic sequence:  
 \begin{align}\label{Mprocess}
    \widetilde X_{k}^h &= \frac{a}{h}\int_{kh}^{(k+1)h}W_s^{H_1}ds+\frac{b}{h}\int_{kh}^{(k+1)h}W_s^{H_2}ds \\ \nonumber
    &= a X_{kh}^{h,H_1} + b X_{kh}^{h,H_2},
    \qquad k\in \mathbb{N}_0,   \, h>0,
\end{align} 
 where $\mathbb{N}_0=\mathbb{N} \cup \left\{0 \right\}$. Here coefficients  $a$ and $b$ are two real numbers, and \(W^{H_1}, W^{H_2}\) are two independent fractional Brownian motions with Hurst parameters $H_i\in(0,1)$, $i=1,2$.
So, we are working with equidistant observations, because one of our goals is to    estimate the parameters $\{H_1, H_2, a, b\}$ of the model using  these observations.
\begin{remark} As usual, taking into account the symmetry of the Gaussian processes $W^{H_1}$  and $W^{H_2}$, we can assume that $a>0,\, b>0.$ Otherwise, in what follows we can just estimate $|a|$  and $|b|.$ 
    
\end{remark}

Of course, the sequence $\{ \widetilde X_{k}^h, \  k\in \mathbb{N}_0\}$
is Gaussian but not stationary. Again, as in Section \ref{sec2}, in order to construct statistical estimators of all parameters basing on  the ergodic theorem, we want to transform it into  a stationary Gaussian stochastic sequence.  This can be done in several ways, but the easiest way is to create an incremental process of the form  
 $\Delta \widetilde X_k^h:=\widetilde X_{k+1}^h-\widetilde X_{k}^h$. Thus, we have
\begin{align*}\label{Dprocess} \Delta \widetilde X_k^h&= a(X^{h,H_1}_{(k+1)h}-X^{h,H_1}_{kh})+b (X^{h,H_2}_{(k+1)h}-X^{h,H_2}_{kh}) \\&=\frac{a}{h}\int_{kh}^{(k+1)h}\bigg(W_{s+h}^{H_1}-W_s^{H_1}\bigg)ds+\frac{b}{h}\int_{kh}^{(k+1)h}\bigg(W_{s+h}^{H_2}-W_s^{H_2}\bigg)ds\\&=:a\Delta X_{k}^{h,H_1}+b\Delta X_{k}^{h,H_2}.
\end{align*}

\begin{proposition}\label{lemmastat}
The Gaussian sequence $\{\Delta \widetilde X_k^h, k\in \mathbb{N}_0\}$
is stationary, with $\bE[\Delta \widetilde X_k^h ] = 0$ for all $k \in \bN_0$ and  
\begin{equation}\label{cov_sumoftwo_increments}
    \bE\left[\Delta \widetilde X_k^h \Delta \widetilde X_{k+n}^h\right]=a^2 h^{2H_1}\gamma( H_1,n)+b^2  h^{2H_2}\gamma( H_2,n),
\end{equation} 
 for all $k,n \in \bN_0$. 
 \end{proposition}
\begin{proof}
Since   $\{\Delta \widetilde X_k^h, k\in \mathbb{N}_0\}$ is Gaussian, and $\bE[\Delta \widetilde X_k^h ] = 0$ for all $k \in \bN_0$, it is sufficient to prove that its    covariance  $\bE\left[\Delta \widetilde X_k^h \Delta \widetilde X_{k+n}^h\right]$ does not depend on $k$. To prove this, we take into account independence and zero mean of $\Delta X_{k}^{h,H_1}$ and $\Delta X_{k}^{h,H_2}$, and state that 

\begin{equation*} 
    \bE\left[\Delta \widetilde X_k^h \Delta \widetilde X_{k+n}^h\right]=a^2 \bE\left[\Delta  X_k^{h,H_1}\Delta  X_{k+n}^{h,H_1}\right]   +b^2 \bE\left[\Delta  X_k^{h,H_2}\Delta  X_{k+n}^{h,H_2}\right].  
\end{equation*}
Now, in order to finish the proof, it is sufficient to note, that, in terms of  \eqref{covfin}, $$\bE\left[\Delta  X_k^{h,H}\Delta  X_{k+n}^{h,H}\right]= \bE[(X_{t+h}^h- X_t^h)(X_{t+(n+1)h}^h- X_{t+nh}^h)],$$ with $t=kh$, and apply Theorem \ref{theoremfour}.
\end{proof}

\subsection{Asymptotic behavior of the covariance}

To assess whether we can use the ergodic theorems \cite[Chapter 2]{Gikhman04} for Gaussian sequences, we need  the following property:   the covariance function \eqref{cov_sumoftwo_increments}  tends to zero  as $n\to \infty$. It is sufficient to establish that for any $H\in(0,1)$,  $\gamma (H, n) $ tends to zero as $n\to \infty$. It is the subject of the next lemma.

\begin{lemma}\label{gammas}
For any $H\in(0,1)$ the sequence $\gamma (H, n), n\in \bN_0 $ defined by equality \eqref{gammy}, is a asymptotically negligible, i.e., tends to zero as $n\to \infty$. More precisely,
\[
\frac{n^{2-2H}\gamma (H, n)}{2H(2H - 1) }\to  1,
\]
as $n\to\infty$.
 \end{lemma}
\begin{proof}
Let us rewrite $\gamma (H, n)$   for $n \geq 2$  by taking out the multiplier  $n^{2H+2}$:
\[
\gamma (H,n)=\frac{n^{2H+2}}{2(2H+1)(2H+2)}g(n^{-1}),
\]
where
\[
g(x)=\left(\!\left(1-2x\right)^{2H+2}\!-4\left(1-x\right)^{2H+2}+6\!-4\left(1+x\right)^{2H+2}+\!\left(1+2x\right)^{2H+2}\right).
\]
The asymptotic behavior of \(\gamma (H,n)\) can be deduced  from the Taylor expansion of $g(x)$ at the origin. Indeed, let us compute the subsequent derivatives of $g(x)$:
\begin{gather*}
g'(x) = (2H+2) \left( -2(1 - 2x)^{2H+1} + 4(1 - x)^{2H+1} - 4(1 + x)^{2H+1} + 2(1 + 2x)^{2H+1} \right),
\\
g''(x) = (2H+2)(2H+1) \left( 4(1 - 2x)^{2H} - 4(1 - x)^{2H} - 4(1 + x)^{2H} + 4(1 + 2x)^{2H} \right),
\\
\begin{aligned}
g'''(x) &= (2H+2)(2H+1)(2H) 
\\
&\qquad\times\left( -8(1 - 2x)^{2H-1} + 4 (1 - x)^{2H-1} - 4 (1 + x)^{2H-1} +8 (1 + 2x)^{2H-1} \right),
\end{aligned}
\\
\begin{aligned}
g^{(iv)}(x) &= (2H+2)(2H+1)(2H)(2H-1)
\\
&\qquad\times\left( 16(1 - 2x)^{2H-2} - 4(1 - x)^{2H-2} - 4(1 + x)^{2H-2} +16 (1 + 2x)^{2H-2} \right).
\end{aligned}
\end{gather*}
Observe that $g(0)=g'(0)=g"(0)=g'''(0)=0$, and the first  
non-zero term in the Taylor expansion of $g$ at zero is the fourth derivative.
\[
g^{(iv)}(0) = 24(2H+2)(2H+1)(2H)(2H-1) 
\]
Then we get from the Taylor expansion that 
\[
g(x)=(2H+2)(2H+1)(2H)(2H-1)x^4 +o(x^4).
\]
Hence, as $n\to\infty,$   $\gamma (H, n)$  behaves as $2H(2H-1)n^{2H-2}$, which implies that  
\[
\frac{ n^{2 - 2H}\gamma (H, n)}{2H(2H - 1)}\to      1, \,n\to \infty.\qedhere
\]
\end{proof}
\begin{remark}
    
It is interesting to note that the covariance of the increments of the process $\{X_{kh}^h,k\ge 1\}$ exhibits the same asymptotic behavior as that of a fractional Brownian motion. Specifically, when $H>\frac{1}{2}$,  the function $\gamma^H(n)$ is strictly positive   and decays slowly, whereas for $H<\frac{1}{2}$,  $\gamma^H(n)$ is strictly negative and exhibits a faster decay. This can be easily explained if for $n\ge 2$ we transform the incremental covariance as follows: \begin{multline*}
     \frac{1}{h^2}\bE\int_{0}^{h}\big(W_{s+h}^{H}-W_s^{H}\big)ds     \int_{nh}^{(n+1)h}\big(W_{t+h}^{H}-W_t^{H}\big)dt\\= \frac{1}{h^2}\int_{0}^{h}\!\!\int_{0}^{h}\bE\big(W_{s+h}^{H}-W_s^{H}\big)       \big(W_{t+(n+1)h}^{H}-W_{t+nh}^{H}\big)dsdt,
     \end{multline*}
\looseness=1 and recall that the increment of fBm are positively (negatively) correlated if $H>1/2$ ($H<1/2$). Moreover, applying the expansion of  $(1+x)^{2H}$, we get that  
\begin{align*}
\MoveEqLeft \bE\big(W_{s+h}^{H}-W_s^{H}\big)       \big(W_{t+(n+1)h}^{H}-W_{t+nh}^{H}\big)\\
&=\frac12\bigg[(t-s+(n+1)h)^{2H}+(t-s+(n-1)h)^{2H}-2(t-s+nh)^{2H} \bigg]\\
&=\frac{(nh)^{2H}}{2}
 \bigg[\bigg(1+\frac{t-s+h}{nh}\bigg)^{2H}+\bigg(1+\frac{t-s-h}{nh}\bigg)^{2H}-2\bigg(1+\frac{t-s}{nh}\bigg)^{2H}  \bigg]\\
&= \frac{H(2H-1)(nh)^{2H}}{2}
 \bigg[ \bigg(\frac{t-s+h}{nh}\bigg)^{2}+\bigg( \frac{t-s-h}{nh}\bigg)^{2 }-2\bigg( \frac{t-s}{nh}\bigg)^{2 }   + o(n^{-2})  \bigg]\\
&\sim H(2H-1)h^{2H} n^{2H-2},
 \end{align*}
and we can apply the Lebesgue dominated convergence theorem in order to get the same asymptotic behavior   of standard and integrated fBm.
Figures \ref{H03} and \ref{H07} illustrate the comparison  of the incremental covariance of the   integrated  and standard  fractional Brownian motions for $H=0.3$ and $H=0.7$, respectively.  One more observation: it follows from Lemma \ref{H0}, (ii),  that for non-overlapping intervals and for $H=\frac12$ the respective increments are independent. Therefore, in this case $g_{1/2}(\frac1n)=0, n\ge 2$, where \[
g_{1/2}(x)=\left(\!\!\left(1-2x\right)^3\!-4\left(1-x\right)^{3}+6\!-4\left(1+x\right)^{3}+\!\left(1+2x\right)^{3}\!\right).
\]
Moreover, equality $g_{1/2}(x)=0$  can be easily checked for any $x\le\frac12$.  \end{remark}

\begin{figure}[ht]
    \centering
    \begin{minipage}{0.47\textwidth}\label{fig1}
        \centering
        \includegraphics[width=\textwidth]{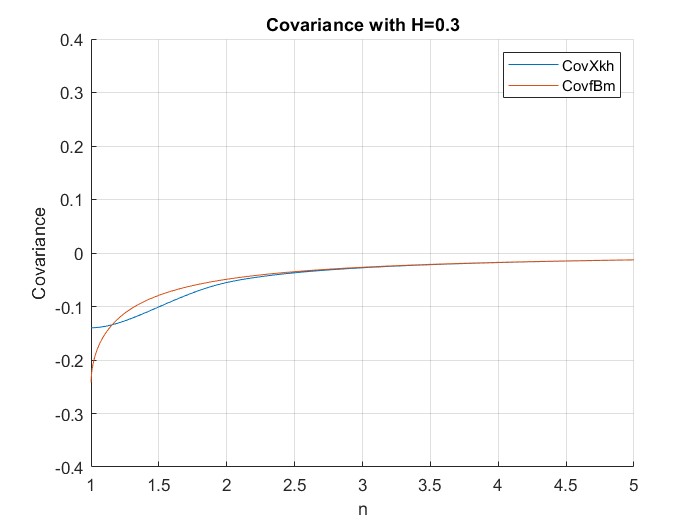}
        \caption{Comparison of the covariance of the increments in our process $\Delta X^h_k$ and the fractional Brownian motion with $H=0.3$.}
        \label{H03}
    \end{minipage}
    \hfill
    \begin{minipage}{0.47\textwidth}\label{fig2}
        \centering
        \includegraphics[width=\textwidth]{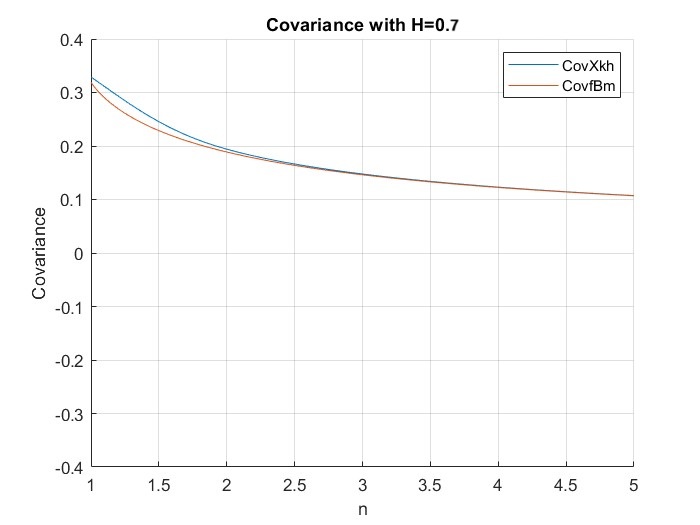}
        \caption{Comparison between the covariance of the increments in our process $\Delta X^h_k$ and the fractional Brownian motion with $H=0.7$.}
        \label{H07}
    \end{minipage}
\end{figure}

\subsection{Ergodic theorem for the normalized integral process}

Taking into account that the normalized integral process is Gaussian, Proposition  \ref{lemmastat} that guarantees its stationarity, Lemma \ref{gammas} that guarantees that the increments are asymptotically uncorrelated, and Theorem 5, Section  8   \cite{Gikhman04}, we can formulate the following theorem.
\begin{theorem}\label{ergodic} 
The Gaussian sequence $\Delta \widetilde X_k^h:=\widetilde X_{ k+1 }^h-\widetilde X_{k}^h, k \in \bN_0$ is ergodic, i.e., for any function $Q\colon\bR\to \bR$ such that $ \bE \left[Q(\Delta \widetilde X_0^h)\right]<\infty$, we have that 
$$\frac1N\sum_{k=0}^{N-1} Q(\Delta \widetilde X_k^h)\to \bE \left[Q(\Delta \widetilde X_0^h)\right] $$  a.s., as $N\to\infty$. 
    
\end{theorem}
\begin{remark} Of course, the same is true for the sequence  $\Delta  X_k^h:=  X_{(k+1)h }^h-  X_{kh}^h$, $k \in \bN_0$ with any Hurst index $H$. 
\end{remark}

\section{Mixed nifBm with drift: model description and drift parameter estimation}
\label{sec4}

Starting from this section, we focus on the statistical inference for a process defined as a linear combination of two independent nifBms with a nonlinear drift component. As a special case, we also consider a related model driven by a single nifBm. This section is devoted to the estimation of the drift parameter, while the estimation of the noise parameters will be addressed in Sections~\ref{constr}--\ref{sec:asnorm}.

  \subsection{Drift parameter estimation in the model with drift and two nifBms}
Let us begin by describing the statistical model. Fix $h > 0$, and let $X^{h,H_i} = \{X_t^{h,H_i},\, t \geq 0\}$, $i = 1,2$, denote two independent nifBms with Hurst parameters $H_1, H_2 \in (0,1)$, as defined in \eqref{Xintmean}. We define the mixed nifBm with drift $Y = \{Y_t,\, t \geq 0\}$ by
\begin{equation}\label{driftTwoProc}
Y_t = \mu G(t) + a X^{h,H_1}_t + b X^{h,H_2}_t,
\end{equation}
where $\mu, a, b \in \mathbb{R}$ are constants, and $G$ is a known deterministic function.  

Suppose that the process $Y$ is observed at discrete time points $t_k = k h$, for $k = 0, \dots, N$. The corresponding observations ${Y_1, \dots, Y_N}$ are given by
\begin{equation}\label{discr-model}
Y_k^h = \mu G(k h) + a X^{h,H_1}_{k h} + b X^{h,H_2}_{k h}, \quad k = 0, \dots, N.
\end{equation}
Denoting $G_k^h \coloneqq G(kh)$ and using the notation \eqref{Mprocess}, we can rewrite \eqref{discr-model} as follows:
\begin{equation}\label{discr-model1}
Y_k^h = \mu G_k^h + \widetilde X_k^h, \quad k = 0, \dots, N.
\end{equation}

In what follows, we assume that $G_0^h = 0$ and $G_k^h \ne 0$ at least for one $k$. Also, in this section, we assume that $H_1, H_2 \in (0,1)$ and the coefficients $a, b \in \mathbb{R} \setminus \{0\}$ are known. The goal is to estimate the drift parameter $\mu \in \mathbb{R}$ based on the discrete observations. Estimation of the remaining parameters will be discussed in Section~\ref{constr}.

Define the vector of increments
$\Delta Y^{(N)} = \left(\Delta Y_0, \dots, \Delta Y_{N-1}\right)^\top$,
where, for $k = 0, \dots, N-1$,
\[
\Delta Y_k = Y_{(k+1) h} - Y_{k h} \\
= \mu \Delta G^h_k + \Delta \widetilde X^h_k,
\]
and
$\Delta G^h_k = G^h_{k+1} - G^h_k$, 
$\Delta\widetilde X^h_k = \widetilde X^h_{k+1} - \widetilde X^h_k$.

Clearly, the vector $\Delta Y^{(N)}$ is Gaussian with distribution $\mathcal{N}(\mu \Delta G^{(N)}, \widetilde\Sigma^{(N)})$, where
$\Delta G^{(N)} = \left(\Delta G^h_0, \dots, \Delta G^h_{N-1}\right)^\top$,
and $\widetilde\Sigma^{(N)}$ is the covariance matrix of the Gaussian vector
\[
\Delta \widetilde{X}^{(N)} \coloneqq \left( \Delta \widetilde X^h_0, \dots, \Delta \widetilde X^h_{N-1} \right)^\top.
\]
Due to the independence of $X^{h,H_1}$ and $X^{h,H_2}$, the covariance matrix $\widetilde\Sigma^{(N)}$ can be  decomposed as
\[
\widetilde\Sigma^{(N)} = a^2 \Sigma^{(N)}_{H_1} + b^2 \Sigma^{(N)}_{H_2},
\]
where $\Sigma^{(N)}_{H_i}$ is the covariance matrix of $(\Delta X^{h,H_i}_0, \dots, \Delta X^{h,H_i}_{N-1})^\top$, $i = 1,2$. By Proposition~\ref{prop:covmatrnonsing}, both $\Sigma^{(N)}_{H_1}$ and $\Sigma^{(N)}_{H_2}$ are positive definite, and hence $\Sigma^{(N)}$ is positive definite and non-singular whenever $a^2 + b^2 > 0$.

Following the methodology of \cite{Mishura2018NonlinearDriftMLE}, the maximum likelihood estimator (MLE) of the drift parameter $\mu$ is given by
\begin{equation}\label{drift_est}
\widehat{\mu}_N = \frac{(\Delta G^{(N)})^\top \left(\widetilde\Sigma^{(N)}\right)^{-1} \Delta Y^{(N)}}{(\Delta G^{(N)})^\top \left(\widetilde\Sigma^{(N)}\right)^{-1} \Delta G^{(N)}}.
\end{equation}

\begin{proposition}\label{prop:drift-mle}
Let $a \ne 0$, $b \ne 0$, and $H_1, H_2 \in (0,1)$.
\begin{enumerate}
\item The MLE $\widehat{\mu}_N$ is unbiased and normally distributed with variance
\begin{equation}\label{variance_drift2}
\Var(\widehat{\mu}_N) = \frac{1}{(\Delta G^{(N)})^\top \left(\widetilde\Sigma^{(N)}\right)^{-1} \Delta G^{(N)}}.
\end{equation}
\item If
\begin{equation}\label{assumpMLEcons2}
\frac{N^{H_1 \vee H_2}}{G^h_N} \to 0 \quad \text{as } N \to \infty,
\end{equation}
then the estimator $\widehat{\mu}_N$ is consistent both in $L^2$ and almost surely.
\end{enumerate}
\end{proposition}

\begin{proof}
1. Since the process $Y$ is Gaussian and the estimator $\widehat{\mu}_n$ is a linear functional of the observations, it is normally distributed. From the model representation
\[
\Delta Y^{(N)} = \mu \Delta G^{(N)} + \Delta \widetilde{X}^{(N)},
\]
it follows that
\[
\widehat{\mu}_N = \mu + \frac{\left(\Delta G^{(N)}\right)^\top \left(\widetilde\Sigma^{(N)}\right)^{-1} \Delta \widetilde{X}^{(N)}}{\left(\Delta G^{(N)}\right)^\top \left(\widetilde\Sigma^{(N)}\right)^{-1} \Delta G^{(N)}},
\]
which shows that $\widehat{\mu}_N$ is unbiased and has variance given by \eqref{variance_drift2}.

2. According to \cite[Theorem 1]{Mishura2018twoest}, a sufficient condition for both $L^2$-consistency and strong consistency of $\widehat{\mu}_N$ is that
\[
\frac{\Var\left(\widetilde X^h_N\right)}{\left(G^h_N\right)^2} \to 0 \quad \text{as } N \to \infty.
\]
This condition holds under assumption \eqref{assumpMLEcons2}, since
\[
\Var\left(\widetilde X^h_N\right) 
= \Var\left(a X^{h,H_1}_{Nh} + b X^{h,H_2}_{Nh}\right)
\sim a^2 (Nh)^{2H_1} + b^2 (Nh)^{2H_2}, \quad \text{as } N \to \infty,
\]
see Remark~\ref{rem:varX}.
\end{proof}

\begin{remark}\label{rem:observations}
In this section, we have assumed that the process $Y$ is observed on an equidistant grid $\{h, 2h, \dots, Nh\}$. However, the same results for the MLE of the drift parameter $\mu$ remain valid for any grid of the form $\{k_1 h, k_2 h, \dots, k_N h\}$, where $k_1 < k_2 < \dots < k_N$. The $L^2$- and strong consistency of the corresponding MLE holds under assumption \eqref{assumpMLEcons2}, provided that $k_N \to \infty$ as $N \to \infty$.

It is important to note that, in contrast, the ergodic-type estimators of $H_1$, $H_2$, $a$, and $b$, which will be introduced in Section~\ref{constr}, do rely crucially on the equidistant nature of the observation grid. Therefore, for consistency with the assumptions in the subsequent sections, we have chosen to adopt the equidistant observation scheme in the present section as well.
\end{remark}

Let us consider a simpler alternative estimator based on two observations only:
\begin{equation}\label{est-drift-last}
\widetilde\mu_N = \frac{Y^h_N - Y^h_0}{G^h_N}.
\end{equation}
Observe that this estimator can be interpreted as the MLE of $\mu$ based on two observations $Y^h_0$ and $Y^h_N$.  
Therefore, taking into account Remark~\ref{rem:observations} and Proposition~\ref{prop:drift-mle}, we immediately obtain the following result.

\begin{proposition}\label{prop:drift-last}
Let $a \ne 0$, $b \ne 0$, and $H_1, H_2 \in (0,1)$.
\begin{enumerate}
\item The estimator $\widetilde \mu_N$ is unbiased and normal with
\begin{equation}\label{var-altdrift}
\Var (\widetilde \mu_N) = \frac{\Var\left(\widetilde X^h_N - \widetilde X^h_0\right)}{\left(G^h_N\right)^2}.
\end{equation}
\item Under assumption \eqref{assumpMLEcons2}, the estimator $\widetilde{\mu}_N$ is consistent both in $L^2$ and almost surely.
\end{enumerate}
\end{proposition}

\begin{remark}
\begin{enumerate}
    \item The MLE $\widehat\mu_N$ is more efficient than the estimator $\widetilde\mu_N$, in the sense that $\Var(\widehat\mu_N) \le \Var(\widetilde\mu_N)$. This result follows from \cite[Corollary~1]{Mishura2018NonlinearDriftMLE}, where the variances of MLEs based on samples of different sizes are compared.

    \item A key advantage of the estimator $\widetilde\mu_N$ is its independence from the noise parameters $H_1$, $H_2$, $a$, and $b$, which may be unknown in practice.

    \item Under the stronger assumption that $\frac{G^h_N}{N}\ge c>0, N\ge 1$,  the strong consistency of $\widetilde\mu_N$ can be derived from the ergodic Theorem~\ref{ergodic}. Indeed, this theorem implies that
    \[
    \frac{\widetilde X_N^h - \widetilde X^h_0}{N} = \frac{1}{N} \sum_{k=0}^{N-1} \Delta \widetilde X_k^h \to \mathbb{E}\left[\Delta \widetilde X_0^h\right] = 0
    \quad \text{a.s., as } N \to \infty,
    \]
    which yields, under the stated assumption, that 
    $\widetilde\mu_N = \mu + \frac{N}{G^h_N} \cdot \frac{\widetilde X_N^h - \widetilde X^h_0}{N}
    \to \mu$, a.s., as $N \to \infty$.
\end{enumerate}
\end{remark}

\subsection{Drift parameter estimation in the model with drift and one nifBm}
Let us consider a particular case of the model \eqref{discr-model} with $b = 0$. In this setting, the statistical model takes the form
\begin{equation}\label{driftOneProc}
Y_k^h = \mu G^h_k + a X^{h,H}_{kh},
\end{equation}
where $X^{h,H}$ is a nifBm with Hurst parameter $H \in (0,1)$, as defined in \eqref{Xintmean}.

For the process \eqref{driftOneProc}, the MLE of the drift parameter $\mu$ is given by
\begin{equation}\label{drift_est1}
\widehat{\mu}_N = \frac{\left(\Delta G^{(N)}\right)^\top \left(\Sigma^{(N)}_H\right)^{-1} \Delta Y^{(N)}}{\left(\Delta G^{(N)}\right)^\top \left(\Sigma^{(N)}_H\right)^{-1} \Delta G^{(N)}},
\end{equation}
where $\Sigma^{(N)}_H$ denotes the covariance matrix of the vector $(\Delta X^{h,H}_0, \dots, \Delta X^{h,H}_{N-1})^\top$.

The following result summarizes the properties of the estimator \eqref{drift_est1}. It follows directly from the results of \cite{Mishura2018NonlinearDriftMLE, Mishura2018twoest}, using the same approach as in Proposition~\ref{prop:drift-mle}. For this reason, we omit the proof.

\begin{proposition}
Let $a \ne 0$ and $H \in (0,1)$.
\begin{enumerate}
\item The MLE $\widehat{\mu}_N$, defined in \eqref{drift_est1}, is unbiased and normally distributed with variance
\begin{equation}\label{variance_drift}
\Var(\widehat \mu_N) = \frac{a^2}{(\Delta G^{(N)})^\top \left(\Sigma^{(N)}_H\right)^{-1}  \Delta G^{(N)}}.
\end{equation}

\item Suppose that
\begin{equation}\label{assumpMLEcons}
\frac{N^H}{G^h_N} \to 0 \quad \text{as } N \to \infty.
\end{equation}
Then the MLE $\widehat{\mu}_N$ is consistent both in the $L^2$-sense and almost surely.
\end{enumerate}
\end{proposition}

The following result gives the properties of an alternative estimator in the model with drift and one nifBm.
\begin{proposition}\label{prop:drift-last-oneproc}
Let $a \ne 0$ and $H \in (0,1)$.
\begin{enumerate}
\item The estimator 
\begin{equation}\label{drift-last1}
\widetilde{\mu}_N = (Y^h_N - Y^h_0)/G^h_N
\end{equation}
is unbiased and normally distributed with variance
\[
\Var(\widetilde\mu_N) = a^2\bE [(X_{Nh}^{h,H} - X_{0}^{h,H})^2] / (G^h_N)^2.
\]

\item Under assumption \eqref{assumpMLEcons}, then the estimator $\widetilde{\mu}_N$ is consistent both in the $L^2$-sense and almost surely.
\end{enumerate}
\end{proposition}

\section{Construction of a strongly consistent estimators for the parameters of (mixed) nifBm}
\label{constr}

In the previous section, we focused on estimating the drift parameter 
$\mu$ in statistical models \eqref{discr-model1} and \eqref{driftOneProc} driven by two or one nifBms, respectively.
In the present section, we address the simultaneous estimation of the parameters of the noise term. Specifically, we construct strongly consistent estimators for these parameters, relying on the ergodic Theorem~\ref{ergodic}.

For technical convenience, we first consider the case in which the drift parameter $\mu$ is known and equals zero. However, our methodology can be extended to the general case of an unknown drift. In that situation, $\mu$ should first be estimated using the estimator $\widetilde\mu_N$ introduced in \eqref{est-drift-last}. Further details of this extension for the models with two and one nifBms will be provided in Remarks \ref{rem:model-with-drift} and \ref{rem:model1-with-drift}, respectively.

\subsection{Estimation of the parameters \texorpdfstring{$H_1$, $H_2$, $a^2$, and $b^2$}{H1, H2, a2, and b2} in the model with two nifBms}

Now, we aim to estimate the vector parameter 
\(\theta = (H_1, H_2, a^2, b^2)\) of the process  
$$\widetilde X^h_k = a X_{kh}^{h,H_1} + b X_{kh}^{h,H_2}.$$
That is, we assume that the drift parameter $\mu$ in model~\eqref{discr-model1} is equal to zero. 
If this assumption does not hold, the parameter $\mu$ should first be estimated using the estimator $\widetilde{\mu}$ from~\eqref{est-drift-last}. 
Subsequently, one may apply the estimator
$\widehat{\theta}_N = (\widehat{H}_{1,N}, \widehat{H}_{2,N}, \widehat{a}^2_N, \widehat{b}^2_N)$
defined in Theorem~\ref{stat}, where the observations $\widetilde{X}^h_k$ are replaced by
\begin{equation}\label{discr-model-approx}
\widetilde{Y}^h_k \coloneqq Y^h_k - \widetilde{\mu}_N G^h_k
= \widetilde X^h_k + \left(\mu - \widetilde{\mu}_N \right) G^h_k.
\end{equation}
This approach requires additional assumptions on the sequence $\{G^h_k, k \ge 1\}$, which are detailed in Remark~\ref{rem:model-with-drift}.

Having four parameters to estimate, we apply Theorem~\ref{ergodic} to construct four corresponding statistics. The first statistic, $\xi^1_N$, is based directly on the stationary increment process
$\Delta \widetilde{X}^h_k = \widetilde{X}^h_{k+1} - \widetilde{X}^h_k$,
which was considered in Theorem~\ref{ergodic}. The remaining three statistics are constructed analogously from the increments
$\Delta \widetilde{X}^{jh}_k = \widetilde{X}^{jh}_{k+1} - \widetilde{X}^{jh}_k$, $j = 2,4,8$,
whose stationarity can be established in the same way as for $\Delta \widetilde{X}^h_k$. Thus, we define the following family of statistics: 
\begin{equation}\label{xi_est}
\xi^j_N := \frac{1}{N} \sum_{k=0}^{N-1} 
\left( \widetilde{X}^{jh}_{k+1} - \widetilde{X}^{jh}_{k} \right)^2   
= \frac{1}{N} \sum_{k=0}^{N-1} \left( \Delta \widetilde{X}^{jh}_{k} \right)^2, 
\quad j = 1,2,4,8.
\end{equation}

All four statistics are constructed in the same manner, differing only in the time step used to compute the increment of the process~\eqref{Mprocess}. This design ensures that the exponents appearing in the estimating equations are powers of two, which simplifies the solution for the estimator $\widehat{\theta}$ of the parameter vector~$\theta$.
\begin{lemma}\label{lemmaest}
The following a.s.\ convergences hold for \(N\xrightarrow{}\infty\):
\begin{equation}\label{manyj}
    \xi^j_N\rightarrow \mathbb{E}(\xi^j_1) = A_1a^2\big(2^{2H_1}-1\big)j^{2H_1} + A_2b^2\big(2^{2H_2}-1\big)j^{2H_2} , \end{equation}
where $ j=1,2,4,8$ and
\[
A_i=\frac{2h^{2H_i}}{(2H_i+1)( H_i+1)},  \quad i=1,2.
\]
\end{lemma}

\begin{proof}
Since the random sequence   $\{\Delta \widetilde X_{k}^h,\, k\ge 0\}$  is  stationary and ergodic for any $h>0$,  the proof of all a.s.\ convergence follows directly from the ergodic Theorem \ref{ergodic} with $Q(x)=x^2$, the independence of the processes $W^{H}_i$, $i=1,2$ and the values of $\mathbb{E}(\xi^j_1)$ that are calculated in accordance with formula \eqref{cov_sumoftwo_increments} from Proposition \ref{lemmastat}, substituting subsequently $h$, $2h$, $4h$, $8h$ in it. 
\end{proof}

\begin{remark}\label{rem:model-with-drift}
For $\mu \ne 0$, the random sequence $\{\Delta Y^h_k,\, k\ge 0\}$, where $Y^h_k$ is defined by \eqref{discr-model1}, is not ergodic.
However, in this case an analogue of Lemma~\ref{lemmaest} can be obtained.
For the modified observations \eqref{discr-model-approx}, define
\[
\widetilde{\xi}^j_N := \frac{1}{N} \sum_{k=0}^{N-1} \left( \widetilde{Y}^{jh}_{k+1} - \widetilde{Y}^{jh}_{k} \right)^2
= \frac{1}{N} \sum_{k=0}^{N-1} \left(\Delta \widetilde{Y}^{jh}_{k} \right)^2, \quad  j \in \{1,2,4,8\}.
\]
Then
\begin{align}
\widetilde{\xi}^j_N 
&= \frac{1}{N} \sum_{k=0}^{N-1} \left(\Delta \widetilde{X}^{jh}_{k} + \left(\mu - \widetilde{\mu}_N \right) \Delta G^{jh}_k \right)^2
\notag\\
&= \xi^j_N +  \left(\mu - \widetilde{\mu}_N \right)^2 \frac{1}{N}\sum_{k=0}^{N-1} \left(\Delta G^{jh}_k \right)^2
+ \left(\mu - \widetilde{\mu}_N \right) \frac{2}{N}  \sum_{k=0}^{N-1} \Delta \widetilde{X}^{jh}_{k}  \Delta G^{jh}_k .
\label{xi-modified}
\end{align}
If assumption \eqref{assumpMLEcons2} holds, then $\widetilde{\mu}_N \to \mu$ a.s.\ as $N \to \infty$, by Proposition~\ref{prop:drift-last}.
Therefore, under the additional assumption
\begin{equation}\label{assump-increments}
\sum_{k=0}^{N-1} \left(\Delta G^{jh}_k \right)^2 = O(N),
\quad N \to \infty,
\end{equation}
we obtain the almost sure convergence
\begin{equation}\label{xi-modified-conv}
 \widetilde{\xi}^j_N \to A_1a^2\big(2^{2H_1}-1\big)j^{2H_1} + A_2b^2\big(2^{2H_2}-1\big)j^{2H_2},
 \quad N \to \infty.
\end{equation}
Indeed, the first term on the right-hand side of \eqref{xi-modified} converges to the stated limit by \eqref{manyj}; 
the second term tends to zero due to \eqref{assump-increments}, and the third term also vanishes, 
since it can be bounded via the first two terms using the Cauchy--Schwarz inequality.

The convergence in \eqref{xi-modified-conv} guarantees the strong consistency, in the model with drift, 
of the estimator $\widehat{\theta}_N = (\widehat{H}_{1,N}, \widehat{H}_{2,N}, \widehat{a}^2_N, \widehat{b}^2_N)$
defined in Theorem~\ref{stat}, where $\widetilde\xi^j_N$ are used instead of $\xi^j_N$.

Assumption \eqref{assump-increments} is satisfied, for instance, if the sequence 
$\{\Delta G^{jh}_k, k \ge 0\}$ is bounded; in particular, for a linear drift function $G(t) = t$,
we have $\Delta G^{jh}_k = jh$ for all $k$ (in this case, condition \eqref{assumpMLEcons2} also holds).
More generally, \eqref{assump-increments} is valid whenever $G$ satisfies a Lipschitz or H\"older condition.
\end{remark}

Now we are in position to construct strongly consistent estimators of the parameter \(\theta=(H_1,H_2,a^2,b^2)^\top\).

The subsequent notation will be used:
$\log x=\log_e x$,
\begin{equation*}
    \log_+ x = 
\begin{cases}
\log  x, & \text{if } x > 0, \\
0, & \text{if } x \leq 0,
\end{cases}
\qquad \sqrt[+]{x} =
\begin{cases}
\sqrt{x}, & \text{if}\, x > 0, \\
0, & \text{if}\, x\, \leq 0.
\end{cases}
\end{equation*}
Also, for any $N\ge 1$  denote
\begin{gather}\label{Dxy}
    D_N = \left(\xi^4_N\xi^2_N-\xi^8_N\xi^1_N\right)^2 - 4\left(\xi^4_N\xi^1_N-\left(\xi^2_N\right)^2\right)\left(\xi^8_N\xi^2_N-\left(\xi^4_N\right)^2\right),
\\
    x_N=\frac{\xi^8_N \xi^1_N-\xi^4_N\xi^2_N+\sqrt[+]{D_N}}{2(\xi^4_N\xi^1_N-(\xi^2_N)^2)}, \qquad y_N=\frac{\xi^8_N \xi^1_N-\xi^4_N\xi^2_N-\sqrt[+]{D_N}}{2(\xi^4_N\xi^1_N-(\xi^2_N)^2)}.
\end{gather} 
From now on, we assume that the fraction equals zero if its denominator equals zero. We also assume that coefficients $a$ and $b$ are non-zero, so that we indeed have two components in the model.

\begin{theorem}\label{stat}
    Let $0< H_2 < H_1 < 1$. The random vector  \(\widehat{\theta}_N=(\widehat{H}_{1,N},\widehat{H}_{2,N},\widehat{a}^2_N, \widehat{b}^2_N )\), where
\begin{gather}\label{estH1}
    \widehat{H}_{1,N}=\frac{1}{2\log2}\log_+\left(x_N \right),
\\
\label{estH2}
    \widehat{H}_{2,N}=\frac{1}{2\log2}\log_+\left(y_N \right),
\\
    \widehat{a}_N^2 =\frac{\bigl(2 \widehat{H}_{1,N} +1\bigr) \bigl( \widehat{H}_{1,N}+1\bigr) \left(\xi_N^2-y_N\xi_N^1\right)}{2h^{ 2\widehat{H}_{1,N}}(x_N-y_N)(x_N-1)},
\shortintertext{and}
    \widehat{b}_N^2 =\frac{\bigl(2\widehat{H}_{2,N}+1\bigr) \bigl(\widehat{H}_{2,N}+1\bigr) \left(\xi_N^2-x_N\xi_N^1\right)}{2h^{2\widehat{H}_{2,N}}(y_N-x_N)(y_N-1)} 
\end{gather}
is a strongly consistent estimator of \(\theta=(H_1,H_2,a^2,b^2)\).
\end{theorem}
\begin{remark} Condition of Theorem \ref{stat} does not mean that we should ``know what index is bigger''. It just means that taking $x_N$ we estimate bigger index, and taking $y_N$ we estimate the smaller one.
\end{remark}

\begin{proof}
    Denote 
    \[x=2^{2H_1}, \quad y=2^{2H_2}, \quad A=a^2A_1, \quad B=b^2A_2, \]
    \[ \eta_j=\mathbb{E}(\xi^j_1),\;j=1,2,4,8.\] 
Then the equalities in the right-hand side of \eqref{manyj} from    Lemma \ref{lemmaest} can be rewritten in the form 
    \begin{equation}\label{eqxi}
        \eta_1=A(x-1)+B(y-1),
    \end{equation}
     \begin{equation}\label{eqpsi}
        \eta_2=Ax(x-1)+By(y-1),
    \end{equation}
    \begin{equation}\label{eqtau}
        \eta_4=Ax^2(x-1)+By^2(y-1),
    \end{equation}
    \begin{equation}\label{eqnu}
        \eta_8=Ax^3(x-1)+By^3(y-1).
    \end{equation}
Solving the system of linear (w.r.t.\ $A$ and $B$)   equations   \eqref{eqxi} and \eqref{eqpsi}, we can present the variables $A$ and $B$ as functions of $x$ and $y$:
\begin{equation}\label{eqAB}
    A=\frac{\eta_2-y\eta_1}{(x-y)(x-1)}, \quad B=\frac{\eta_2-x\eta_1}{(y-x)(y-1)}.
\end{equation}
Inserting into \eqref{eqtau}  expressions \eqref{eqAB} obtained for   $A$ and $B$, we get that   
\begin{equation}\label{16tau}
\begin{aligned}
       \eta_4&=\frac{\eta_2-y\eta_1}{x-y}x^2+\frac{\eta_2-x\eta_1}{y-x}y^2\\
        & =\frac{1}{x-y}\bigl(\eta_2 x^2-\eta_1 x^2y-\eta_2 y^2+\eta_1 xy^2 \bigr)\\
      &=\frac{1}{x-y}\bigl(\eta_2(x^2-y^2)-\eta_1 xy(x-y) \bigr)=\eta_2(x+y)-\eta_1 xy,
\end{aligned}
 \end{equation}
whence 
\begin{equation}\label{x1}
    x=\frac{\eta_4-\eta_2 y}{\eta_2-\eta_1 y}.
\end{equation}
Note that    \(\eta_2-\eta_1 y=A(x-1)(x-y)\), $A>0$ ($A$ is positive and can be zero only if $a=0$ but this is the degenerate case with only one component for the process $X_k^h$), $x=2^{H_1}>0$, $x=1$ if $H_1=0$ and $x-y=0$ if $H_1=H_2$. If we assume $H_1 \neq H_2$, the denominator is non-zero. 

Similarly, it follows from \eqref{eqnu} and \eqref{eqAB} that      
\begin{equation} \label{64nu}
\begin{aligned}
        \eta_8&=\frac{\eta_2-y\eta_1}{x-y}x^3+\frac{\eta_2-x\eta_1}{y-x}y^3\\&
       =\frac{1}{x-y}\bigl(\eta_2 x^3-\eta_1 x^3y-\eta_2 y^3+\eta_1 xy^3 \bigr)\\&
        =\frac{1}{x-y}\bigl(\eta_2\left(x^3-y^3\right)-\eta_1 xy(x^2-y^2) \bigr)\\&
      =\frac{1}{x-y}\bigl(\eta_2(x-y)(x^2+xy+y^2)-\eta_1 xy(x-y)(x+y) \bigr)
      \\&=\eta_2\left(x^2+xy+y^2\right)-\eta_1 xy(x+y).
    \end{aligned}
\end{equation}
    Multiplying the  right-hand side of  \eqref{16tau} by $(x+y)$ and subtracting the right-hand side of  \eqref{64nu}, we get that 
    \begin{equation*}
        \eta_2(x+y)^2-\eta_2\left(x^2+xy+y^2\right)=\eta_4(x+y)-\eta_8, 
        \end{equation*}
 whence            
\begin{equation}\label{x2}
    x=  \frac{\eta_8- \eta_4 y}{\eta_4-\eta_2 y}.
\end{equation}
Again, the denominator is non-zero. 
 It follows from   equations \eqref{x1} and \eqref{x2} that 
    \begin{equation*}
        \frac{\eta_8- \eta_4 y}{\eta_4-\eta_2 y}=\frac{\eta_4-\eta_2 y}{\eta_2-\eta_1 y},
    \end{equation*}
and we obtain the following quadratic equation w.r.t.\ $y$:
\begin{equation}\label{quad}
    y^2\left(\eta_1 \eta_4-\eta_2^2\right)+y(\eta_2 \eta_4 -\eta_1 \eta_8)+\left(\eta_2\eta_8-\eta_4^2\right)=0.
\end{equation}
It is evident that $x$ satisfies the same quadratic equation due to symmetry. Hence, $x$ and $y$ with $(x>y)$ are the two roots of equation \eqref{quad},
\begin{equation}
    x,y=\frac{(\eta_1 \eta_8-\eta_2 \eta_4)\pm\sqrt{D}}{2(\eta_1 \eta_4-\eta_2^2)},
\end{equation}
where the denominator is again   non-zero if $H_1 \neq H_2$, 
and 
\begin{equation}
    \begin{aligned}
            D=&(\eta_2 \eta_4 -\eta_1 \eta_8)^2-4(\eta_1 \eta_4-\eta_2^2)(\eta_2\eta_8-\eta_4^2).
    \end{aligned}
\end{equation}
Let us  check if the discriminant $D$ is nonnegative. We use the equalities \eqref{eqxi}--\eqref{64nu} and conclude that 
\begin{align}\label{2418}
\eta_2 \eta_4 -\eta_1 \eta_8 &= AB(x-1)(y-1)(xy^2+x^2y-y^3-x^3)\\ 
&= -AB(x-1)(y-1)(x+y)(x-y)^2  < 0,
\\
\label{1422}
\eta_1 \eta_4-\eta_2^2 &= AB(x-1)(y-1)(x^2-2xy+y^2)\\
&= AB(x-1)(y-1)(x-y)^2 > 0,
\\
\label{2844}
\eta_2\eta_8-\eta_4^2 &= AB(x-1)(y-1)(xy^3-2xy^2+x^3y)\\
&= AB(x-1)(y-1)xy(x-y)^2 > 0.
\end{align}
Thus
\begin{align*}
D &= A^2B^2(x-1)^2(y-1)^2(x+y)^2(x-y)^4-4A^2B^2(x-1)^2(y-1)^2xy(x-y)^4\\
 &= A^2B^2(x-1)^2(y-1)^2(x-y)^6>0.
\end{align*}
Since $D$ is positive,  we have two different roots
 
\begin{equation}\label{xandy}
    x=\frac{\eta_1 \eta_8-\eta_2 \eta_4+\sqrt{D}}{2(\eta_1 \eta_4-\eta_2^2)}, \qquad y=\frac{\eta_1 \eta_8-\eta_2 \eta_4-\sqrt{D}}{2(\eta_1 \eta_4-\eta_2^2)}
\end{equation}
Since \(x=2^{2H_1}\) and \(y=2^{2H_2}\), we get from \eqref{xandy} that 
\begin{equation}\label{estimH}
  H_1=\frac{1}{2\log2}\log\left(\frac{\eta_1 \eta_8-\eta_2 \eta_4+\sqrt{D}}{2(\eta_1 \eta_4-\eta_2^2)} \right), \quad H_2=\frac{1}{2\log2}\log\left(\frac{\eta_1 \eta_8-\eta_2 \eta_4-\sqrt{D}}{2(\eta_1 \eta_4-\eta_2^2)} \right)
\end{equation}
   Recall that $a^2$ and $b^2$ can be obtained from equations \eqref{eqAB}:
\begin{equation}\label{aNbN}
\begin{aligned}
    &a^2 =\frac{(2H_1+1)(H_1+1)(\eta_2-y\eta_1)}{2h^{2H_1}(x-y)(x-1)},\qquad
  b^2   =\frac{(2H_2+1)(H_2+1)(\eta_2-x\eta_1)}{2h^{2H_2}(y-x)(y-1)}.
\end{aligned}
\end{equation}

Now, in order to obtain the estimators, we substitute $\eta_j $ with  $\xi_N^j$ in \eqref{xandy} and \eqref{estimH} and obtain $D_N$, $x_N$, $y_N$, $\widehat H_{1,N}$ and $\widehat H_{2,N}$. Then we substitute in \eqref{aNbN} $\widehat H_{i,N}$, $i=1,2$ instead of $H_i$, $i=1,2$ and also other estimators and obtain $\widehat{a}^2_N$ and $\widehat{b}^2_N$. Note  that the random denominators are asymptotically non-zero with probability 1. Strong consistency of these estimators follows from strong consistency of $\xi_N^j$. 
\end{proof}
\begin{remark}
    The estimators for \(H_1\) and \(H_2\) do  not depend explicitly on  \(h\), the dependence on $h$ is implicitly present in the estimators \(\xi^j_N\) for \(j=1,2,4,8\).
\end{remark}

\begin{remark}\label{substit}
Note that the values $\widetilde X^{2h}_k$, $\widetilde X^{4h}_k$, and $\widetilde X^{8h}_k$ used in our estimators can be expressed in terms of $\widetilde X^h_k$. By \eqref{Mprocess},  
\begin{align}
\widetilde X^{2h}_{k} &= \frac{1}{2h} \int_{2hk}^{2h(k+1)} \left(a W_u^{H_1} + b W_u^{H_2}\right) du \notag\\
&= \frac12 \left( \frac{1}{h}\int_{2hk}^{2hk+h} \!\left(a W_u^{H_1} + b W_u^{H_2}\right) du 
   + \frac{1}{h} \int_{(2k+1)h}^{(2k+2)h}\! \left(a W_u^{H_1} + b W_u^{H_2}\right) du \right) \notag\\
&= \tfrac12 \left(\widetilde X_{2k}^h + \widetilde X_{2k+1}^h\right).
\label{2h-2processes}
\end{align}
More generally, we obtain  
\begin{equation}\label{XjhviaXh}
\widetilde X^{jh}_{k} = \frac1j \sum_{l=0}^{j-1} \widetilde X_{jk+l}^h,
\quad j \in \{1, 2, 4, 8\}.
\end{equation}
\end{remark}

\begin{remark}\label{newnot}
Note that if we fix a final time horizon $T$ and follow the trajectory of our process up to $T$, the number of observations, that is, the terms in the summations of each statistic $\xi^1$, $\xi^2$, $\xi^4$, $\xi^8$ will be different. Indeed, given that for each statistic we consider increments of different amplitude, we have $8N$ elements for $\xi^1$, $4N$ elements for $\xi^2$, $2N$ elements for $\xi^4$ and $N$ elements for $\xi^8$. Thus, we should replace $\xi^1_N$ with $\xi^1_{8N}$, $\xi^2_N$ with $\xi^2_{4N}$, $\xi^4_N$ with $\xi^4_{2N}$ in \eqref{Dxy}. 
However, this notation will prove more convenient and will therefore be adopted for studying asymptotic normality, which will be the subject of Section~\ref{sec:asnorm}.
Of course, since we use the ergodic theorem and let $N \to \infty$, it does not matter that the statistics involve different numbers of elements.
\end{remark}

\subsection{Estimation of the parameters \texorpdfstring{$H$ and $a^2$}{H and a2} in the model with one nifBm}
\label{ssec:1proc}
Consider now the case when we have only one nifBm, more precisely, let  $b=0$. To distinguish from the previous case, we call this model \emph{``model with one nifBm''}.  In this case  we need only two statistics to compute the estimators for $H$ and $a^2$. Of course, we can combine different pairs of statistics; however, let us consider the simplest possibility. So, we observe the 
increments  $\Delta X_k^{jh} = X^{jh}_{(k+1)h} - X_{kh}^{jh}$   for $j =1, 2$. In this case $$\xi^j_N := \frac{1}{N} \sum_{k=0}^{N-1} \left(   X^{jh}_{(k+1)h} -   X^{jh}_{kh} \right)^2   = \frac{1}{N} \sum_{k=0}^{N-1} \left( \Delta   X^{jh}_{k} \right)^2, \quad  j=1,2,$$ the convergence \eqref{manyj} takes the form 
\begin{equation}\label{manyjonly2}
    \xi^j_N\rightarrow \mathbb{E}(\xi^j_1) = A_1a^2\big(2^{2H}-1\big)j^{2H},\quad
    \text{a.s., as } N\to\infty, \quad j=1,2, \end{equation}
and we immediately obtain the following result, which we formulate without proof, since it uses the same approach as in the proof of Theorem \ref{stat}, but with significant simplifications.

Denote \begin{equation}\label{nivBmone}\widehat H_N=\frac{1}{2\log 2} \log_+ \left(\frac{\xi_N^2}{\xi_N^1}\right), \quad \widehat A_{N}=\frac{2h^{2\widehat H_N}}{\bigl(2\widehat H_N+1\bigr)\bigl(\widehat H_N+1\bigr)}. 
\end{equation}

\begin{theorem}\label{est1H}
 Consider the model with one nifBm, and let \(0<H<1\) and $a^2>0$. Then the estimator  \(\widehat{\theta}_N=(\widehat{H}_N, \widehat{a}_N^2)\)  where $\widehat{H}_N$ is taken from \eqref{nivBmone} and
 \begin{equation}\label{est_a}
   \widehat{a}_N^2=\frac{\xi_N^1}{\widehat A_{N}\bigl(2^{2\widehat H_N}-1\bigr)},  
 \end{equation}
 is a strongly consistent estimator of \(\theta=(H,a^2)\).
\end{theorem}

\begin{remark}\label{rem:model1-with-drift}
Similarly to the model with two nifBms, the proposed approach can be adapted to the model \eqref{driftOneProc} with a non-zero drift.
Under assumptions \eqref{assumpMLEcons} and \eqref{assump-increments} on the sequence ${G^h_k}$, we employ a two-stage procedure for the strongly consistent estimation of all model parameters.
First, the drift parameter $\mu$ is estimated using the strongly consistent estimator $\widetilde{\mu}_N$ from Proposition \ref{prop:drift-last-oneproc}.
Subsequently, $H$ and $a^2$ are estimated using \eqref{nivBmone} and \eqref{est_a}, respectively, where the observations $X^{jh}_{kh}$ are replaced by $Y^{jh}_k - \widetilde{\mu}_N G^{jh}_k$.
\end{remark}

\begin{remark}\label{rem:modification}
Proceeding analogously to Remark \ref{newnot}, we note that the statistic $\xi^1_N$ appearing in the estimators \eqref{nivBmone} and \eqref{est_a} 
should be replaced by $\xi^1_{2N}$. 
\end{remark}

\section{Asymptotic normality}
\label{sec:asnorm}

In this section, we establish the joint asymptotic normality of the estimators. 
We first consider the model with a single nifBm and without drift, and prove the asymptotic normality of the estimator
\(\widehat{\theta}_N=(\widehat{H}_N, \widehat{a}_N^2)\)
defined in \eqref{nivBmone}--\eqref{est_a}. 
For this case, detailed proofs are provided. 
We then present, without proofs, the corresponding results for more general models that incorporate a drift term and/or 
an additional nifBm in the noise component. 
These results can be obtained by applying essentially the same arguments, although their proofs require considerably more technical details and lead to cumbersome expressions; for this reason, they are omitted here.

\subsection{Asymptotic normality in the model with one nifBm}

Now we prove the asymptotic normality of the strongly consistent estimator for $\theta = (H,a^2)$ in the one-nifBm model (see subsection~\ref{ssec:1proc}). Following Remark~\ref{rem:modification}, we employ the modified notation
for the statistics,
$\xi^1_{2N}$ and $\xi^2_N$. By~\eqref{manyjonly2}, as $N \to \infty$,
\begin{equation}\label{convtof}
\xi^1_{2N} \to \mathbb{E}(\xi^1_1) = A_1 a^2 (2^{2H}-1)
\eqqcolon f_1(\theta),
\quad 
\xi^2_{N} \to \mathbb{E}(\xi^2_1) = A_1 a^2 (2^{2H}-1) 2^{2H}
\eqqcolon f_2(\theta)
\;\;\text{a.s.}
\end{equation}
The next theorem establishes their joint asymptotic normality.

\begin{theorem}\label{asnorm}
Let $0 < H < \tfrac{3}{4}$. Consider the vectors $\bar \xi_N=(\xi^1_{2N}, \xi^2_N)^\top$, $\bar \eta=(\mathbb{E}(\xi^1_1), \mathbb{E}(\xi^2_1))^\top$. Then $\bar\xi_N$ is asymptotically normal, that is,
\begin{equation*}
 \sqrt{N}(\bar \xi_N-\bar \eta)   
=\sqrt{N}\begin{pmatrix}
\xi^1_{2N} - \mathbb{E}(\xi^1_1) \\
\xi^2_N - \mathbb{E}(\xi^2_1) 
\end{pmatrix}
\xrightarrow{d}
\mathcal{N}\bigl(0,\widetilde \Sigma\bigr),
\quad \text{as } N \to \infty, 
\end{equation*}
where the asymptotic covariance matrix $\widetilde \Sigma$ is given by 
\[
\widetilde \Sigma =
\begin{pmatrix}
\widetilde\Sigma_{11} & \widetilde\Sigma_{12} \\
\widetilde\Sigma_{12} & \widetilde\Sigma_{22}
\end{pmatrix},
\]
with entries 
\begin{gather}
\widetilde \Sigma_{11}= h^{4H} \sum_{i=-\infty}^{\infty} {\gamma}(H,i)^2,
\qquad
\widetilde \Sigma_{22}=2^{4H+1}h^{4H} \sum_{i=-\infty}^{\infty} \gamma(H,i)^2,\\
\widetilde \Sigma_{12}= h^{4H}\sum_{i=-\infty}^{\infty} \gamma(H,i)\bigl(3\gamma(H,i) + 4\gamma(H,i+1) + \gamma(H,i+2)\bigr).
\end{gather}
\end{theorem}

\begin{remark}\label{rem:convCov}
For $H < \tfrac{3}{4}$, all series defining the entries of the asymptotic covariance matrix converge. Moreover, for any $\alpha, \beta \in \mathbb{Z}$,
\begin{equation}
    \sum_{i=-\infty}^{\infty} \left|\gamma(H,i+\alpha)\, \gamma(H,i+\beta)\right| < \infty.
\end{equation}
This follows from the Cauchy--Schwarz inequality:
\[
\sum_{i=-\infty}^{\infty} \left|\gamma(H,i+\alpha)\, \gamma(H,i+\beta)\right| 
\leq \sqrt{\sum_{i=-\infty}^{\infty} \gamma(H,i+\alpha)^2 \sum_{i=-\infty}^{\infty} \gamma(H,i+\beta)^2} 
= \sum_{i=-\infty}^{\infty} \gamma(H,i)^2 < \infty,
\]
where convergence of the final series is guaranteed by Lemma \ref{gammas}.
\end{remark}

\begin{proof}
The proof consists of two parts. In the first part, we establish that $ \sqrt{N}(\bar \xi_N-\bar \eta)$ converges to a bivariate normal distribution. In the second part, we compute the entries of the asymptotic covariance matrix.

\emph{Step 1: Asymptotic normality.}
Recall that
\[
\xi^1_{2N} = \frac{1}{2N} \sum_{k=0}^{2N-1} \bigl(\Delta X^{h}_{k}\bigr)^2, \qquad
\xi^2_N = \frac{1}{N} \sum_{k=0}^{N-1} \bigl(\Delta X^{2h}_{k}\bigr)^2,
\]
where $\Delta X^h_k = X^h_{(k+1)h}-X^h_{kh}$.  
We express $\{\Delta X^{2h}_k\}$ in terms of $\{\Delta X^h_k\}$.  
Similarly to \eqref{2h-2processes}, we have
\[
X^{2h}_{kh} = \tfrac{1}{2} \bigl(X_{2kh}^h + X_{(2k+1)h}^h\bigr),
\]
which implies
\begin{align}
\Delta X^{2h}_k &=
 \tfrac{1}{2}  \bigl( X^h_{(2k+3)h}+ X^h_{(2k+2)h}-X^h_{(2k+1)h}-X^h_{2kh} \bigr) \\
&= \tfrac{1}{2}\bigl( \Delta X^h_{2k+2}+2 \Delta X^h_{2k+1} + \Delta X_{2k}^h \bigr).\label{realizations}
\end{align}
Define $Y_k=(Y_k^{(1)},Y_k^{(2)}, Y_k^{(3)})$, $k \in \mathbb N_0$, by
\begin{equation}\label{Yk}
Y_k^{(1)}=\Delta X^h_{2k}, \qquad Y_k^{(2)}= \Delta X^h_{2k+1}, \qquad Y_k^{(3)}=\Delta X_{2k+2}^{h}.
\end{equation}
Then
\[
\xi^{1}_{2N} = \frac{1}{2N}\sum_{k=0}^{N-1}  \Bigl(\bigl(Y^{(1)}_k\bigr)^2+\bigl(Y^{(2)}_k\bigr)^2\Bigr), 
\quad
\xi^{2}_N= \frac{1}{4N} \sum_{k=0}^{N-1} \bigl(Y_k^{(1)}+ 2 Y_k^{(2)}+Y_k^{(3)}\bigr )^2.
\]

To prove asymptotic normality of $(\xi^{1}_{2N}, \xi^{2}_N)^\top$, we use the Cramér--Wold device.  
Let $\alpha, \beta \in \mathbb R$ with $\alpha^2+\beta^2 \neq 0$ and define
\[
F(y) = \frac{\alpha}{2}(y_1^2+y_2^2) + \frac{\beta}{4}(y_1+2y_2+y_3)^2, 
\qquad y=(y_1,y_2,y_3)^\top \in \mathbb R^3.
\]
Then
\[
\alpha \xi^1_{2N} + \beta \xi^2_N= \frac{1}{N} \sum_{k=0}^{N-1} F(Y_k).
\]
Hence we need to show that
\begin{equation}\label{fNormal}
\sqrt N \Bigl(\alpha(\xi^1_{2N}- \bE[\xi^1_1]) +\beta(\xi^2_N-\bE [\xi^2_1])\Bigr)
=\frac{1}{\sqrt N}\sum_{k=0}^{N-1} (F(Y_k)-\bE[F(Y_k)])
\end{equation}
converges in distribution to a Gaussian law.  
This will be proved by applying the Breuer--Major theorem for stationary Gaussian vectors \cite{Arcones1994, Nourdin2011}; see Theorem \ref{thm:Arcones} in the Appendix.  

Since $F(Y_1)$ is a quadratic form of a Gaussian vector, its Hermite rank is $q = 2$ (see Example \ref{ex:quad-form} in Appendix~\ref{app:gaus-vect}). Therefore, to verify the condition \eqref{cond-arcones} of Theorem \ref{thm:Arcones}
it suffices to establish
\begin{equation}\label{arcones-condition}
\sum_{n \in \mathbb Z} \bigl(\cov(Y_1^{(p)}, Y^{(l)}_{1+n})\bigr)^2 < \infty, 
\quad p,l \in \{1,2,3\}.
\end{equation}
Using \eqref{Yk} and the definition of $\gamma(H,n)$, the covariance structure is
\begin{align*}
\cov(Y_1^{(1)},Y^{(1)}_{1+n}) &= h^{2H}\gamma(H,2n), & 
\cov(Y^{(1)}_{1}, Y^{(2)}_{1+n}) &= h^{2H}\gamma(H,2n+1),\\
\cov(Y^{(1)}_{1}, Y^{(3)}_{1+n}) &= h^{2H}\gamma(H,2n+2), &
\cov(Y^{(2)}_{1}, Y^{(1)}_{1+n}) &= h^{2H}\gamma(H,2n-1),\\
\cov(Y^{(2)}_{1}, Y^{(2)}_{1+n}) &= h^{2H}\gamma(H,2n), &
\cov(Y^{(2)}_{1}, Y^{(3)}_{1+n}) &= h^{2H}\gamma(H,2n+1),\\
\cov(Y^{(3)}_{1}, Y^{(1)}_{1+n}) &= h^{2H}\gamma(H,2n-2), &
\cov(Y^{(3)}_{1}, Y^{(2)}_{1+n}) &= h^{2H}\gamma(H,2n-1),\\
\cov(Y^{(3)}_{1}, Y^{(3)}_{1+n})&=h^{2H}\gamma(H,2n).&&
\end{align*}
Hence, the condition \eqref{arcones-condition} follows from Remark \ref{rem:convCov}. For example,
\[
\sum_{n=- \infty}^{\infty} \bigl(\cov(Y_1^{(1)},Y^{(1)}_{1+n})\bigr)^2
= h^{4H}\sum_{n=- \infty}^{\infty} \gamma(H,2n)^2,
\]
which converges for $H<\tfrac{3}{4}$ by Remark \ref{rem:convCov}. The other cases are analogous. Thus, the conditions of Theorem~\ref{thm:Arcones} are satisfied, proving that \eqref{fNormal} converges weakly to a centered Gaussian distribution.

\medskip

\emph{Step 2: Identification of the asymptotic covariance matrix.}
To determine the covariance matrix of the limiting law, we analyze
\[
N \Var(\xi^1_{2N}), \quad N \Var(\xi^2_N), \quad 
N \cov(\xi^1_{2N},\xi^2_N).
\]

Using the stationarity of $\{\Delta X^h_{k}\}_k$ and formula \eqref{cov-squares}, we obtain
\[
\cov\bigl((\Delta X^h_k)^2,(\Delta X^h_j)^2 \bigr)
=2h^{4H} \gamma(H,k-j)^2.
\]
Hence
\begin{align*}
N \Var(\xi^1_{2N})&= \frac{1}{4N} \Var\Bigl(\sum_{k=0}^{2N-1} (\Delta X_k^h)^2\Bigr)= \frac{1}{4N}\sum_{k=0}^{2N-1} \sum_{j=0}^{2N-1} \cov\Bigl( (\Delta X^h_k)^2,(\Delta X^h_j)^2 \Bigr)\\
&= \frac{h^{4H}}{2N} \sum_{k=0}^{2N-1} \sum_{j=0}^{2N-1} \gamma(H,k-j)^2.
\end{align*}
Upon rearrangement of the sum, we get
\begin{align} N \Var(\xi^1_{2N})&= \frac{h^{4H}}{2N}\sum_{k=0}^{2N-1} \sum_{i=k-2N+1}^{k} {\gamma}(H,i)^2 \\
&= \frac{h^{4H}}{2N} \sum_{i=-2N+1}^{0} \sum_{k=0}^{2N-1+i} {\gamma}(H,i)^2 + \frac{h^{4H}}{2N} \sum_{i=1}^{2N-1} \sum_{k=i}^{2N-1} {\gamma}(H,i)^2\\
&= h^{4H}\sum_{i=-2N+1}^{2N-1} \Bigl( 1-\frac{|i|}{2N}\Bigr) {\gamma}(H,i)^2 \to h^{4H} \sum_{i=-\infty}^{\infty} {\gamma}(H,i)^2,
\label{var1a}
\end{align}
as $N \to \infty$, where the interchange of limit and summation is justified by the dominated convergence theorem, as ensured by Lemma \ref{gammas}.

Analogous computations yield
\[
N \Var(\xi^2_N) \to 2^{4H+1}h^{4H} \sum_{i=-\infty}^{\infty} \gamma(H,i)^2,
\quad \text{as } N\to \infty.
\]

Let us now compute the asymptotic behavior of the covariance between $\xi^1_{2N}$ and $\xi^2_N$. We have
\begin{align*} N \cov \left(\xi^1_{2N}, \xi^2_N\right) &= N \cov \left( \frac{1}{2N} \sum_{k=0}^{2N-1}(\Delta X_k^h)^2, \frac{1}{N} \sum_{j=0}^{N-1}(\Delta X_j^{2h})^2\right) \\ &= \frac{1}{8N}\cov\left(\sum_{k=0}^{N-1} \left((\Delta X_{2k}^h)^2+ (\Delta X^h_{2k+1})^2\right), \sum_{j=0}^{N-1}(\Delta X^h_{2j+2}+2\Delta X^h_{2j+1}+\Delta X_{2j}^h )^2\right) \\ &= \frac{1}{8N}\sum_{k=0}^{N-1} \sum_{j=0}^{N-1} \cov\Bigl((\Delta X_{2k}^h)^2+ (\Delta X^h_{2k+1})^2, (\Delta X^h_{2j+2}+2\Delta X^h_{2j+1}+\Delta X_{2j}^h )^2\Bigr)\\ &= \frac{1}{4N}\sum_{k=0}^{N-1} \sum_{j=0}^{N-1} \biggl[\Bigl(\cov\left(\Delta X_{2k}^h,\Delta X^h_{2j+2}+2\Delta X^h_{2j+1}+\Delta X_{2j}^h \right) \Bigr)^2\\ &\qquad\qquad\qquad\qquad+\Bigl(\cov\left(\Delta X^h_{2k+1},\Delta X^h_{2j+2}+2\Delta X^h_{2j+1}+\Delta X_{2j}^h \right) \Bigr)^2\biggr], \end{align*}
where we used formula \eqref{cov-squares}.

Expressing the covariances in terms of $\gamma(H,n)$ via relation \eqref{cov-deltaX}, we obtain
\begin{align*} 
\MoveEqLeft N \cov (\xi^1_{2N}, \xi^2_N)\\*
&= \frac{h^{4H}}{4N}\sum_{k=0}^{N-1} \sum_{j=0}^{N-1} \Bigl[\bigl(\gamma(H,2k-2j-2) + 2\gamma(H,2k-2j-1) + \gamma(H,2k-2j)\bigr)^2\\ 
&\qquad\qquad\qquad\qquad+\bigl(\gamma(H,2k-2j-1) + 2\gamma(H,2k-2j) + \gamma(H,2k-2j+1)\bigr)^2\Bigr].
\end{align*}

Proceeding as in \eqref{var1a}, by rearranging the sums and passing to the limit as $N \to \infty$, we obtain
\begin{align*}
N \cov (\xi^1_{2N}, \xi^2_N)
&\to\frac{h^{4H}}{2}\sum_{i=-\infty}^{\infty}\Bigl[\bigl(\gamma(H,2i-2) + 2\gamma(H,2i-1) + \gamma(H,2i)\bigr)^2\\
&\qquad\qquad\qquad\qquad+\bigl(\gamma(H,2i-1) + 2\gamma(H,2i) + \gamma(H,2i+1)\bigr)^2\Bigr] \\
&= \frac{h^{4H}}{2}\sum_{j=-\infty}^{\infty}\bigl(\gamma(H,j) + 2\gamma(H,j+1) + \gamma(H,j+2)\bigr)^2. 
\end{align*}

Finally, expanding the square and exploiting the shift-invariance of the sums yields the expression for $\widetilde \Sigma_{12}$ stated in the theorem.

This completes the proof.
\end{proof}

We are now in a position to prove the asymptotic normality of $\widehat \theta_N=(\widehat H_N,\widehat a^2_N)^\top$.
Recall that now we consider the estimator defined by \eqref{nivBmone}--\eqref{est_a}, where $\xi^1_N$ is replaced with $\xi^2_N$, that is,
\[
\widehat H_N=\frac{1}{2\log 2} \log \left(\frac{\xi_N^2}{\xi_{2N}^1}\right),
\quad
\widehat{a}_N^2=\frac{\xi_{2N}^1 \bigl(2\widehat H_N+1\bigr) \bigl(\widehat H_N+1\bigr)}{2h^{2\widehat H_N} \bigl(2^{2\widehat H_N}-1\bigr)}. 
\]
Let $f(\theta) = (f_1(\theta), f_2(\theta))^\top$, where $f_1(\theta)$ and $f_2(\theta)$ are defined in \eqref{convtof}.
Then, by construction, the estimator $\widehat\theta_N$ has the form
$\widehat\theta_N = q(\bar\xi_N)$,
where $\bar \xi_N=(\xi^1_{2N}, \xi^2_N)^\top$ and $q = f^{(-1)}$ is the inverse function of $f$, derived in Theorem \ref{est1H}.
\begin{theorem}\label{deltamethod}
Let $0<H< \frac{3}{4}.$ The estimator $\widehat \theta_N=(\widehat H_N,\widehat a^2_N)^\top$ is asymptotically normal, that is
$$\sqrt N(\widehat \theta_N -\theta)=
\sqrt{N} 
\begin{pmatrix}
\widehat H_N - H \\
\widehat a^2_N - a^2 
\end{pmatrix}
\xrightarrow{d} \mathcal{N}\bigl(0, \Sigma_0\bigr),
$$
with the asymptotic covariance matrix $\Sigma_0$ given by
$$
\Sigma_0 = \bigl(f'(\theta)\bigr)^{-1} \, \widetilde \Sigma \, 
(\bigl(f'(\theta)\bigr)^{-1})^\top,
$$
where $\widetilde \Sigma$ is defined in Theorem \ref{asnorm} and  $f'$ is the Jacobian matrix of $f$ (the elements of this matrix are computed in the proof).
\end{theorem}

\begin{proof}
We apply delta method \cite{Kubilinus2017} for function $q = f^{(-1)}$ and sequence $\bar \xi_N$ for which we proved asymptotic normality in Theorem \ref{asnorm}.
To apply the delta method, it is necessary to compute the derivative matrix of $q$ and verify its nonsingularity. However, due to the complexity of $q$, we instead calculate the derivative matrix of $f$ at $\theta$, since by the inverse function theorem, $q' = (f')^{-1}$. Showing that $f'(\theta)$ is nonsingular thus ensures that $q'$ exists and is nonsingular, allowing the application of the delta method.
Thus, we first evaluate the Jacobian matrix of $f(\theta)$. We have
$$f_1(\theta)=
\frac{2 h^{2H}}{(2H+1)(H+1)} a^2 \bigl(2^{2H} - 1 \bigr),
$$
hence
\begin{equation}
\frac{\partial f_1}{\partial H}=
2 a^2 h^{2H}\frac{
\left[
2 (2^{2H} - 1) \log h 
+  2^{2H+1} \log2
\right]
(2H+1)(H+1)
-
(2^{2H} - 1) (4H+3)
}{
\bigl((2H+1)(H+1)\bigr)^2
}
\end{equation}
and
\begin{equation}
   \frac{\partial f_1}{\partial a^2}= \frac{2 h^{2H}}{(2H+1)(H+1)} \bigl(2^{2H} - 1 \bigr) .
 \end{equation}
Consider now
$$
f_2(\theta) = \frac{2 a^2 h^{2H} 2^{2H} (2^{2H}-1)}{(2H+1)(H+1)}.
$$
Let
$$
N(\theta) = 2 a^2 h^{2H} 2^{2H} (2^{2H}-1), \qquad 
D(\theta) = (2H+1)(H+1).
$$
We have
\begin{gather*}
\frac{\partial D}{\partial H} = \big((2H+1)(H+1)\big)
= 2(H+1) + (2H+1) = 4H + 3,
\\
 \frac{\partial N}{\partial H}= 2 a^2 \left[
h^{2H} 2^{2H} (2 \log h + 2 \log 2)(2^{2H}-1)
+ h^{2H} 2^{2H} \big(2 \log 2 \cdot 2^{2H}\big)
\right],
\end{gather*}
whence,
\begin{multline*}
\frac{\partial f_2}{\partial H} = 
\frac{2 a^2 h^{2H} 2^{2H}}
{((2H+1)(H+1))^2}
\\
\times\left[
\big((2\log h+2\log 2)(2^{2H}-1)+2\log 2 \cdot 2^{2H}\big)(2H+1)(H+1)
-(2^{2H}-1)(4H+3)
\right]
.
\end{multline*}
and
\begin{equation}
\frac{\partial f_2}{\partial a^2}= \frac{h^{2H} 2^{2H+1} (2^{2H}-1)}{(2H+1)(H+1)}.
\end{equation}
By computing the determinant of the Jacobian matrix
$f' = \left(\begin{smallmatrix}
\partial f_1/\partial H & \partial f_1/\partial a^2 \\
\partial f_2/\partial H & \partial f_2/\partial a^2
\end{smallmatrix}\right)$, we obtain
$$
\det\big(f'({\theta})\big) = 
- \frac{a^2  h^{4H} 2^{2H+3}  \log 2 \big(2^{2H} - 1\big)^2}{(2H+1)^2(H+1)^2}.
$$

It is easy to see that the determinant is strictly negative for all positive $a$, $h$, and $H$. This guarantees the invertibility of the matrix $f'(\theta)$ in the considered parameter domain.
The application of the delta method concludes the proof.
\end{proof}

\subsection{Asymptotic normality in the model with two nifBms}
We now extend Theorems~\ref{asnorm} and \ref{deltamethod} to the case of a linear combination of two processes. The proof is omitted, as it follows the same arguments as in the two-dimensional case, differing only in the increased algebraic complexity.  

A strongly consistent estimator of the parameter $\theta = (H_1, H_2, a^2, b^2)$ was established in Theorem~\ref{stat}. Analogously to the model with a single nifBm, we now consider a modified estimator $\widehat\theta_N$, where the statistics $\xi^1_N$, $\xi^2_N$, and $\xi^4_N$ are replaced by $\xi^1_{8N}$, $\xi^2_{4N}$, and $\xi^4_{2N}$, respectively; see Remark~\ref{substit}. Accordingly, we introduce  
\[
\bar \xi_N=(\xi^1_{8N}, \xi^2_{4N}, \xi^4_{2N}, \xi^8_N)^\top, 
\qquad 
\bar \eta=(\mathbb{E}(\xi^1_1), \mathbb{E}(\xi^2_1), \mathbb{E}(\xi^4_1), \mathbb{E}(\xi^8_1))^\top =: f(\theta),
\]
where the expectations $\mathbb{E}(\xi^j_1)$, $j=1,2,4,8$, are given by~\eqref{manyj}.  

The following results establish the asymptotic normality of $\bar \xi_N$ and $\widehat \theta_N$ in this framework. We do not provide the explicit form of the asymptotic covariance matrix $\widetilde \Sigma$, as its derivation is analogous to that in the single-fBm case but leads to extremely cumbersome expressions.

\begin{theorem}\label{asnorm2}
Let $0<H_1<H_2<\frac{3}{4}$. Then $\bar\xi_N$ is asymptotically normal, that is
\begin{equation*}
 \sqrt{N}(\bar \xi_N-\bar \eta)   
=\sqrt{N}\begin{pmatrix}
\xi^1_{8N} - \mathbb{E}(\xi^1_1) \\
\xi^2_{4N} - \mathbb{E}(\xi^2_1) \\
\xi^4_{2N} - \mathbb{E}(\xi^4_1) \\
\xi^8_N - \mathbb{E}(\xi^8_1) 
\end{pmatrix}
\xrightarrow{d}
\mathcal{N}\bigl(0,\widetilde \Sigma\bigr),
\quad \text{as } N\to\infty.
\end{equation*}
Moreover, the estimator $\widehat \theta_N=(\widehat H_{1,N},\widehat H_{2,N},\widehat a^2_N,\widehat b^2_N)^\top$ is asymptotically normal, that is
    $$\sqrt N(\widehat \theta_N -\theta)=
\sqrt{N} 
\begin{pmatrix}
\widehat H_{1,N} - H \\
\widehat H_{2,N} - H \\
\widehat a^2_N - a^2 \\
\widehat b^2_N - b^2 
\end{pmatrix}
\xrightarrow{d} \mathcal{N}\bigl(0, \Sigma_0\bigr),
\quad \text{as } N\to\infty.
$$
with the asymptotic covariance matrix $\Sigma_0$ given by
$$
\Sigma_0 = \bigl(f'(\theta)\bigr)^{-1} \, \widetilde \Sigma \, 
(\bigl(f'(\theta)\bigr)^{-1})^\top,
$$
where $f'$ is the Jacobian matrix of $f$.
\end{theorem}

\section{Computational analysis}
\label{sec:simul}

\looseness=-1 This section presents the results of our computational analysis.  
In Subsection~\ref{ssec:simul-drift}, we examine the behavior of the drift parameter estimators for the model with a single nifBm process (equation \eqref{driftOneProc}), and then extend the estimation procedure to the model with two nifBm processes (equation \eqref{discr-model}).  

In Subsection~\ref{ssec:simul-hurst}, we conduct a simulation study to assess the performance of the estimators for the Hurst parameters and the diffusion coefficients.  
This analysis is divided into two parts. First, we evaluate the estimators $\widehat{H}_N$ (equation~\eqref{nivBmone}) and $\widehat{a}_N^2$ (equation~\eqref{est_a}) in the case $b=0$, which corresponds to the model with one nifBm. Second, we validate the estimators stated in Theorem~\ref{stat} for the vector parameter \(\theta=(H_1,H_2,a^2,b^2)^\top\).

For all estimators, we report their sample means and empirical standard deviations, along with theoretical standard deviations where available. In the tables, we use the notation $\mean{\cdot}$, $\sige{\cdot}$, and $\sigt{\cdot}$ to denote the mean, empirical standard deviation, and theoretical standard deviation, respectively.

\subsection{Simulation study for the drift parameter estimators}
\label{ssec:simul-drift}

We investigate the performance of the drift estimators for both models: the one driven by a single nifBm and the one driven by two nifBm processes. Specifically, we fix the true parameter value $\mu = 4$ and consider the function
$G(t) = 5\cos t - e^{-4t} + 2t^2$,
which satisfies condition~\eqref{assumpMLEcons2} for arbitrary values of the Hurst exponents. The diffusion coefficients are set to $a=b=1$.

The sample paths of the processes are generated using the Cholesky decomposition method \cite{Coeurjolly2000}, which exploits the covariance structure of the underlying Gaussian processes.
The implementation is carried out in \textsc{Matlab} using the built-in \textsc{chol} function.

\paragraph{Drift estimation in a model with one nifBm.}  
We begin with the estimation of the drift parameter in the model \eqref{driftOneProc}, which involves a single nifBm with drift. The estimators $\widehat{\mu}_N$ and $\widetilde{\mu}_N$, defined in \eqref{drift_est1} and \eqref{drift-last1}, respectively, are evaluated for several values of the Hurst parameter, $H \in \{0.1, 0.3, 0.5, 0.7, 0.9\}$, and for two values of the averaging window, $h \in \{2, 2^2\}$.  

The discrete-time increments process,
\[
\Delta X_k^h = X_{(k+1)h}^h - X_{kh}^h,
\]
is simulated using the Cholesky method, based on the analytical expression of the covariance matrix given in equation~\eqref{gammy}. Once realizations of $\Delta X_k^h$ are generated, the increments of the function $G(t)$ are added to obtain the process $\Delta Y^{(N)}$.  

For each configuration $(H,h)$, we conduct 100 simulations. For every simulated trajectory, the estimator $\widehat{\mu}_N$ is computed according to \eqref{drift_est1} for various values of the number of observations $N$. The empirical mean and standard deviation are then calculated across the replications. In addition, the theoretical standard deviations are derived from formula~\eqref{variance_drift}.  

We also examine the performance of the alternative estimator $\widetilde{\mu}_N$ by comparing its empirical standard deviation, obtained from simulations, with the corresponding theoretical value. The latter is given by the square root of the variance of $\widetilde{\mu}_N$, computed as
\begin{align}
\Var(\widetilde{\mu}_N) 
= \frac{a^2 h^{2H}}{(G_N^h)^2(2H+1)(2H+2)}
\Big[(N+1)^{2H+2} + (N-1)^{2H+2} - 2N^{2H+2} - 2\Big],
\end{align}
which follows from Proposition~\ref{prop:drift-last-oneproc}, using \eqref{covfar} and \eqref{var1}.

\begin{table}
\centering
\footnotesize
\setlength{\tabcolsep}{4pt}
\caption{Drift estimators $\widehat{\mu}_N$ and $\widetilde{\mu}_N$ in the model with one nifBm.}
\label{tab:drift1}
\begin{tabular}{llcccccc}
\toprule
& & \multicolumn{3}{c}{$h=2$} & \multicolumn{3}{c}{$h=2^2$} \\
\cmidrule(lr){3-5}\cmidrule(lr){6-8}
$H$ & & $N=2^3$ & $N=2^5$ & $N=2^7$ & $N=2^3$ & $N=2^5$ & $N=2^7$ \\
\midrule
\multirow{6}{*}{0.1}
& $\mean{\widehat{\mu}}$     & 3.99976 & 4.00004 & 4.00000 & 3.99945 & 4.00001 & 4.00000 \\
& $\sige{\widehat{\mu}}$     & 0.00510 & 0.00053 & 0.00003 & 0.00561 & 0.00051 & 0.00000 \\
& $\sigt{\widehat{\mu}}$     & 0.00559 & 0.00048 & 0.00004 & 0.00599 & 0.00052 & 0.00000 \\
& $\mean{\widetilde{\mu}}$   & 4.00045 & 3.99995 & 4.00001 & 4.00095 & 4.00004 & 3.99999 \\
& $\sige{\widetilde{\mu}}$   & 0.00651 & 0.00045 & 0.00004 & 0.00568 & 0.00057 & 0.00005 \\
& $\sigt{\widetilde{\mu}}$   & 0.00733 & 0.00058 & 0.00004 & 0.00786 & 0.00062 & 0.00005 \\
\midrule
\multirow{6}{*}{0.3}
& $\mean{\widehat{\mu}}$     & 4.00152 & 3.99994 & 3.99998 & 3.99881 & 3.99994 & 4.00002 \\
& $\sige{\widehat{\mu}}$     & 0.01344 & 0.00140 & 0.00013 & 0.01625 & 0.00166 & 0.00018 \\
& $\sigt{\widehat{\mu}}$     & 0.01279 & 0.00141 & 0.00014 & 0.01575 & 0.00173 & 0.00017 \\
& $\mean{\widetilde{\mu}}$   & 3.99817 & 3.99969 & 3.99991 & 3.99663 & 3.99942 & 3.99997 \\
& $\sige{\widetilde{\mu}}$   & 0.01445 & 0.00165 & 0.00015 & 0.01417 & 0.00214 & 0.00019 \\
& $\sigt{\widetilde{\mu}}$   & 0.01676 & 0.00164 & 0.00016 & 0.02063 & 0.00203 & 0.00020 \\
\midrule
\multirow{6}{*}{0.5}
& $\mean{\widehat{\mu}}$     & 3.99952 & 3.99936 & 4.00011 & 3.99924 & 3.99952 & 4.00002 \\
& $\sige{\widehat{\mu}}$     & 0.02284 & 0.00311 & 0.00044 & 0.03116 & 0.00431 & 0.00061 \\
& $\sigt{\widehat{\mu}}$     & 0.02269 & 0.00320 & 0.00042 & 0.03209 & 0.00452 & 0.00059 \\
& $\mean{\widetilde{\mu}}$   & 3.99907 & 3.99925 & 4.00005 & 3.99831 & 3.99984 & 4.00005 \\
& $\sige{\widetilde{\mu}}$   & 0.02269 & 0.00338 & 0.00049 & 0.03375 & 0.00579 & 0.00071 \\
& $\sigt{\widetilde{\mu}}$   & 0.03077 & 0.00388 & 0.00049 & 0.04351 & 0.00548 & 0.00069 \\
\midrule
\multirow{6}{*}{0.7}
& $\mean{\widehat{\mu}}$     & 3.99175 & 4.00051 & 3.99977 & 4.00518 & 4.00263 & 4.00031 \\
& $\sige{\widehat{\mu}}$     & 0.03442 & 0.00606 & 0.00106 & 0.05696 & 0.01022 & 0.00185 \\
& $\sigt{\widehat{\mu}}$     & 0.03659 & 0.00665 & 0.00114 & 0.05944 & 0.01080 & 0.00184 \\
& $\mean{\widetilde{\mu}}$   & 4.00987 & 3.99994 & 3.99991 & 4.00366 & 3.99717 & 3.99962 \\
& $\sige{\widetilde{\mu}}$   & 0.04859 & 0.00904 & 0.00154 & 0.07372 & 0.01335 & 0.00221 \\
& $\sigt{\widetilde{\mu}}$   & 0.05437 & 0.00895 & 0.00148 & 0.08833 & 0.01453 & 0.00241 \\
\midrule
\multirow{6}{*}{0.9}
& $\mean{\widehat{\mu}}$     & 3.99896 & 4.00091 & 3.99983 & 3.99089 & 4.00054 & 4.00060 \\
& $\sige{\widehat{\mu}}$     & 0.04943 & 0.01187 & 0.00243 & 0.08474 & 0.02017 & 0.00450 \\
& $\sigt{\widehat{\mu}}$     & 0.04602 & 0.01090 & 0.00244 & 0.08588 & 0.02034 & 0.00455 \\
& $\mean{\widetilde{\mu}}$   & 3.99131 & 3.99958 & 4.00013 & 3.99674 & 3.99842 & 3.99953 \\
& $\sige{\widetilde{\mu}}$   & 0.07886 & 0.01719 & 0.00436 & 0.16066 & 0.03891 & 0.00819 \\
& $\sigt{\widetilde{\mu}}$   & 0.09515 & 0.02057 & 0.00449 & 0.17756 & 0.03839 & 0.00837 \\
\bottomrule
\end{tabular}
\end{table}

Table~\ref{tab:drift1} reports the means, empirical and theoretical standard deviations of both estimators.
The simulation results align closely with our theoretical findings. Both $\widehat{\mu}_N$ and $\widetilde{\mu}_N$ are essentially unbiased, with empirical means near the true value of drift $\mu = 4$. Empirical standard deviations match the theoretical values and decrease rapidly with increasing sample size $N$, confirming the consistency and asymptotic normality of both estimators.

Overall, $\widehat{\mu}_N$ is slightly more efficient than $\widetilde{\mu}_N$, particularly for small samples and large Hurst exponents $H$. For large $N$, both estimators achieve high precision, with empirical variability closely following theoretical predictions. The alternative estimator $\widetilde{\mu}_N$ shows greater variability, especially for $H$ near one and small $N$, reflecting the loss of efficiency from using only two observations. Nevertheless, both theoretical and empirical standard deviations decrease consistently as $N$ increases.

In summary, the simulations confirm the consistency of both estimators while highlighting the superior efficiency of the maximum likelihood approach.

\paragraph{Drift estimation in a model with two nifBm}  
Within the same framework, we now consider model~\eqref{discr-model1}, which includes two independent nifBms and a drift term. We investigate the performance of the MLE $\widehat{\mu}_N$ of the drift parameter $\mu$ defined in equation~\eqref{drift_est}. In addition, we examine the alternative estimator $\widetilde{\mu}_N$, introduced in equation~\eqref{est-drift-last}, which relies on two observations.  

\looseness=-1 To evaluate the accuracy of both methods, we compare the empirical standard deviations with their theoretical counterparts. For $\widehat{\mu}_N$, a closed-form variance expression is provided in equation~\eqref{variance_drift2}, whereas for $\widetilde{\mu}_N$ the variance can be computed explicitly as $\Var(\widetilde{\mu}_N) = a^2 T_1 + b^2 T_2$, where
\[
T_i = \frac{h^{2H_i}}{(G_N^h)^2(2H_i+1)(2H_i+2)}\!
\left[(N+1)^{2H_i+2}+(N-1)^{2H_i+2}-2N^{2H_i+2}-2\right], 
\quad i=1,2.
\]

As in the single-nifBm case, we consider $N \in \{2^3,2^5,2^7\}$ observation points and two averaging windows, $h \in \{2,2^2\}$. For each configuration, 100 independent simulations are performed.
\begin{table}
\centering
\footnotesize
\setlength{\tabcolsep}{4pt}
\caption{Drift estimators $\widehat{\mu}_N$ and $\widetilde{\mu}_N$ in the model with two nifBms.}
\label{tab:drift_two_nifbm}
\begin{tabular}{lllcccccc}
\toprule
& &  & \multicolumn{3}{c}{$h=2$} & \multicolumn{3}{c}{$h=2^2$} \\
\cmidrule(lr){4-6}\cmidrule(lr){7-9}
$H_1$ & $H_2$ & & $N=2^3$ & $N=2^5$ & $N=2^7$ & $N=2^3$ & $N=2^5$ & $N=2^7$ \\
\midrule
\multirow{6}{*}{0.1} & \multirow{6}{*}{0.3} 
& $\mean{\widehat{\mu}}$ & 3.99938 & 4.00001 & 4.00001 & 3.99951 & 3.99968 & 4.00000 \\
& & $\sige{\widehat{\mu}}$ & 0.01306 & 0.00167 & 0.00015 & 0.01807 & 0.00165 & 0.00017 \\
& & $\sigt{\widehat{\mu}}$ & 0.01399 & 0.00149 & 0.00015 & 0.01687 & 0.00181 & 0.00017 \\
& & $\mean{\widetilde{\mu}}$ & 3.99953 & 4.00016 & 3.99998 & 4.00385 & 4.00009 & 4.00001 \\
& & $\sige{\widetilde{\mu}}$ & 0.00819 & 0.00143 & 0.00016 & 0.02082 & 0.00226 & 0.00022 \\
& & $\sigt{\widetilde{\mu}}$ & 0.00837 & 0.00175 & 0.00017 & 0.02211 & 0.00212 & 0.00020 \\
\midrule
\multirow{6}{*}{0.1} & \multirow{6}{*}{0.5}
& $\mean{\widehat{\mu}}$ & 3.99824 & 4.00003 & 4.00004 & 3.99953 & 3.99963 & 4.00001 \\
& & $\sige{\widehat{\mu}}$ & 0.02263 & 0.00320 & 0.00037 & 0.03172 & 0.00472 & 0.00061 \\
& & $\sigt{\widehat{\mu}}$ & 0.02349 & 0.00325 & 0.00042 & 0.03275 & 0.00456 & 0.00059 \\
& & $\mean{\widetilde{\mu}}$ & 3.99865 & 3.99956 & 4.00010 & 4.00025 & 3.99902 & 4.00005 \\
& & $\sige{\widetilde{\mu}}$ & 0.02565 & 0.00356 & 0.00051 & 0.03316 & 0.00515 & 0.00069 \\
& & $\sigt{\widetilde{\mu}}$ & 0.03163 & 0.00392 & 0.00049 & 0.04421 & 0.00552 & 0.00069 \\
\midrule
\multirow{6}{*}{0.3} & \multirow{6}{*}{0.5}
& $\mean{\widehat{\mu}}$ & 3.99990 & 4.00060 & 3.99993 & 3.99817 & 3.99997 & 3.99993 \\
& & $\sige{\widehat{\mu}}$ & 0.02745 & 0.00850 & 0.00045 & 0.03636 & 0.01463 & 0.00060 \\
& & $\sigt{\widehat{\mu}}$ & 0.02615 & 0.00960 & 0.00044 & 0.03587 & 0.01326 & 0.00067 \\
& & $\mean{\widetilde{\mu}}$ & 3.99882 & 3.99976 & 4.00004 & 3.99829 & 3.99982 & 3.99997 \\
& & $\sige{\widetilde{\mu}}$ & 0.02907 & 0.00385 & 0.00046 & 0.03903 & 0.00551 & 0.00070 \\
& & $\sigt{\widetilde{\mu}}$ & 0.03503 & 0.00421 & 0.00051 & 0.04815 & 0.00585 & 0.00072 \\
\midrule
\multirow{6}{*}{0.3} & \multirow{6}{*}{0.7}
& $\mean{\widehat{\mu}}$ & 3.99509 & 4.00073 & 4.00010 & 4.00340 & 4.00009 & 4.00010 \\
& & $\sige{\widehat{\mu}}$ & 0.03863 & 0.00717 & 0.00116 & 0.06232 & 0.01110 & 0.00207 \\
& & $\sigt{\widehat{\mu}}$ & 0.03947 & 0.00687 & 0.00115 & 0.06226 & 0.01101 & 0.00186 \\
& & $\mean{\widetilde{\mu}}$ & 3.99228 & 4.00021 & 3.99995 & 4.00442 & 4.00138 & 4.00025 \\
& & $\sige{\widetilde{\mu}}$ & 0.05530 & 0.00829 & 0.00159 & 0.07430 & 0.01515 & 0.00222 \\
& & $\sigt{\widetilde{\mu}}$ & 0.05690 & 0.00910 & 0.00149 & 0.09071 & 0.01468 & 0.00241 \\
\midrule
\multirow{6}{*}{0.5} & \multirow{6}{*}{0.7} & $\mean{\widehat{\mu}}$ & 4.00221 & 4.00056 & 4.00014 & 3.99529 & 3.99930 & 4.00007 \\
& & $\sige{\widehat{\mu}}$ & 0.04494 & 0.00830 & 0.00121 & 0.06574 & 0.01188 & 0.00183 \\
& & $\sigt{\widehat{\mu}}$ & 0.04352 & 0.00745 & 0.00122 & 0.06821 & 0.01181 & 0.00194 \\
& & $\mean{\widetilde{\mu}}$ & 3.99382 & 3.96247 & 4.00011 & 4.00016 & 4.00095 & 3.99986 \\
& & $\sige{\widetilde{\mu}}$ & 0.04386 & 0.01974 & 0.00164 & 0.00926 & 0.01521 & 0.00228 \\
& & $\sigt{\widetilde{\mu}}$ & 0.06248 & 0.00127 & 0.00156 & 0.00975 & 0.01554 & 0.00250 \\
\bottomrule
\end{tabular}
\end{table}

Table~\ref{tab:drift_two_nifbm} reports the simulation results for the drift coefficient $\mu$ across different pairs of Hurst parameters $(H_1,H_2)$. It presents the sample means of both estimators along with their empirical and theoretical standard deviations.

The results for the two-process model (Table~\ref{tab:drift_two_nifbm}) are consistent with those of the single-process case (Table~\ref{tab:drift1}). In both settings, the estimators are essentially unbiased, with empirical means very close to the true drift. Regarding variability, the MLE $\widehat{\mu}_N$ is systematically more efficient than the alternative estimator $\widetilde{\mu}_N$, particularly for larger sample sizes. Only in small-$N$ scenarios two estimators exhibit similar performance; as $N$ increases, the superiority of $\widehat{\mu}_N$ becomes evident across all parameter configurations.\enlargethispage{5pt}

This pattern is common to both the single- and two-process models, although the presence of two fractional components leads to higher overall variability, especially when at least one Hurst parameter is large. Thus, while both estimators are consistent, the MLE provides a more reliable and efficient estimation method, with the advantage most pronounced in smoother regimes and for larger samples.

\subsection{Simulation study for the Hurst parameters and diffusion coefficients}
\label{ssec:simul-hurst}

In this section, we investigate the performance of simultaneous estimators for the Hurst exponents and diffusion coefficients for models with one and two nifBm processes. We fix the values of $a$ and $b$, while considering several values of the Hurst exponents.

\paragraph{Estimation of the Hurst parameter and diffusion coefficient in the model with one nifBm.}  
We start with the model $Y_k^h = a X^{h,H}_{kh}$,  where the value of $a$ is set to 1, involving a single nifBm without drift, and evaluate the estimator \(\widehat{\theta}_N=(\widehat{H}_N, \widehat{a}_N^2)\) defined in \eqref{nivBmone} and~\eqref{est_a}.  

For a fixed step size $h$, we generate the increments \(\Delta X^{h}\) using the Cholesky method based on the covariance structure. The realization of \(\Delta X^{2h}\) is then computed according to \eqref{realizations}, so that it incorporates half of the increments present in \(\Delta X^{h}\). From these processes, we compute the statistics \(\xi_{2N}^1\) and \(\xi_N^2\) to study the behavior of \(\widehat{\theta}_N=(\widehat{H}_N, \widehat{a}_N^2)\) introduced in Theorem~\ref{est1H}.  

For each combination of $H$ and $h$, we simulate a trajectory and apply the estimators to obtain estimates of $H$ and $a^2$. We report the empirical mean and standard deviation over 100 replications and, using Theorem~\ref{deltamethod}, the asymptotic theoretical standard deviation.  

Table~\ref{tab:1H_est} presents the performance of the estimators $\widehat H_N$ and $\widehat a^2_N$ for $H \in \{0.1, 0.3, 0.5, 0.7\}$, $h \in \{2, 2^2, 2^4\}$, and sample sizes $N = 2^n$, $n \in \{6, 8, 10, 12\}$. The estimator $\widehat H_N$ shows almost unbiased behavior in all Hurst parameters, with only slight deviations for small $N$ and extreme $H$ values, particularly $H = 0.1$. Its empirical standard deviation $\sige{\widehat H_N}$ decreases rapidly with $N$ and closely matches the theoretical standard deviation $\sigt{\widehat H_N}$ for moderate and large samples, confirming the accuracy of the theoretical variance formula. The increment size $h$ exerts minimal influence on $\widehat H_N$, slightly affecting variability only for small $N$. 

In contrast, the diffusion coefficient estimator $\widehat a^2_N$ is considerably more sensitive to both sample size $N$ and averaging window $h$. For small $N$ and large $h$, it exhibits pronounced bias and elevated variability, especially for small Hurst indices ($H = 0.1$ or $0.3$). As $N$ increases, both bias and variability diminish and the empirical standard deviation converges toward the theoretical value. 

These results indicate that $\widehat H_N$ provides stable and reliable estimates even for moderate $N$, whereas accurate estimation of $\widehat a^2_N$ requires sufficiently large samples and moderate values of $h$.

\begin{table}
\centering
\footnotesize
\setlength{\tabcolsep}{3.1pt}
\caption{Estimators $\widehat H_N$ and $\widehat a^2_N$ in the model with one nifBm.}
\label{tab:1H_est}
\begin{tabular}{lcccccccccccc}
\toprule
 & \multicolumn{4}{c}{$h=2$} & \multicolumn{4}{c}{$h=2^2$} & \multicolumn{4}{c}{$h=2^4$} \\
\cmidrule(lr){2-5}\cmidrule(lr){6-9}\cmidrule(lr){10-13}
 & $N=2^6$ & $2^8$ & $2^{10}$ & $2^{12}$ 
 & $2^6$ & $2^8$ & $2^{10}$ & $2^{12}$ 
 & $2^6$ & $2^8$ & $2^{10}$ & $2^{12}$ \\
\midrule
\multicolumn{13}{c}{$H=0.1$} \\
\midrule
$\mean{\widehat H}$   & 0.0746 & 0.0834 & 0.1022 & 0.1000 & 0.0952 & 0.1002 & 0.0967 & 0.1026 & 0.1096 & 0.1055 & 0.1001 & 0.1015 \\
$\sige{\widehat H}$            & 0.1713 & 0.0833 & 0.0352 & 0.0198 & 0.1483 & 0.0669 & 0.0362 & 0.0195 & 0.1635 & 0.0752 & 0.0357 & 0.0196 \\
$\sigt{\widehat H}$            & 0.1087 & 0.0544 & 0.0272 & 0.0136 & 0.1087 & 0.0544 & 0.0272 & 0.0136 & 0.1087 & 0.0544 & 0.0272 & 0.0136 \\
$\mean{\widehat a^2}$ & 2.0548 & 1.7555 & 1.2109 & 1.0352 & 2.4617 & 2.0531 & 1.7098 & 1.0147 & 3.9263 & 0.8146 & 1.8697 & 1.0452 \\
$\sige{\widehat a^2}$          & 16.830 & 7.0895 & 0.8605 & 0.0867 & 26.357 & 7.1395 & 3.1218 & 0.2407 & 31.873 & 5.3269 & 6.3387 & 0.3355 \\
$\sigt{\widehat a^2}$          & 1.0176 & 0.5088 & 0.2544 & 0.0718 & 1.1663 & 0.5831 & 0.2916 & 0.1458 & 1.4649 & 0.7325 & 0.3662 & 0.1831 \\
\midrule
\multicolumn{13}{c}{$H=0.3$} \\
\midrule
$\mean{\widehat H}$   & 0.2715 & 0.2961 & 0.3000 & 0.3004 & 0.3067 & 0.2948 & 0.3052 & 0.2993 & 0.2875 & 0.2934 & 0.2980 & 0.3004 \\
$\sige{\widehat H}$            & 0.1325 & 0.6090 & 0.0320 & 0.0159 & 0.1377 & 0.0594 & 0.0327 & 0.0153 & 0.1435 & 0.0622 & 0.0358 & 0.0141 \\
$\sigt{\widehat H}$            & 0.1194 & 0.0597 & 0.0298 & 0.0149 & 0.1194 & 0.0597 & 0.0298 & 0.0149 & 0.1194 & 0.0597 & 0.0298 & 0.0149 \\
$\mean{\widehat a^2}$ & 1.1625 & 1.0719 & 1.0082 & 1.0023 & 4.2266 & 1.0948 & 0.9937 & 1.0066 & 0.8646 & 1.2299 & 1.0609 & 1.0041 \\
$\sige{\widehat a^2}$          & 4.3630 & 0.3158 & 0.1328 & 0.0597 & 9.9780 & 0.3462 & 0.1743 & 0.0768 & 7.3627 & 0.7895 & 0.3227 & 0.1095 \\
$\sigt{\widehat a^2}$          & 0.4354 & 0.2177 & 0.1089 & 0.0544 & 0.5904 & 0.2952 & 0.1476 & 0.0738 & 0.9115 & 0.4557 & 0.2279 & 0.1139 \\
\midrule
\multicolumn{13}{c}{$H=0.5$} \\
\midrule
$\mean{\widehat H}$   & 0.4826 & 0.5024 & 0.4990 & 0.4998 & 0.4936 & 0.4957 & 0.4952 & 0.5002 & 0.4981 & 0.4977 & 0.5031 & 0.4979 \\
$\sige{\widehat H}$            & 0.1175 & 0.0523 & 0.0311 & 0.0121 & 0.1161 & 0.0529 & 0.0277 & 0.0118 & 0.1187 & 0.0548 & 0.0252 & 0.0125 \\
$\sigt{\widehat H}$            & 0.1104 & 0.0552 & 0.0276 & 0.0138 & 0.1104 & 0.0552 & 0.0276 & 0.0138 & 0.1104 & 0.0522 & 0.0276 & 0.0138 \\
$\mean{\widehat a^2}$ & 1.1302 & 1.0085 & 1.0113 & 1.0023 & 1.2494 & 1.0489 & 1.0231 & 0.9995 & 1.3550 & 1.0952 & 0.9873 & 1.0181 \\
$\sige{\widehat a^2}$          & 0.4158 & 0.1355 & 0.0729 & 0.0311 & 1.2327 & 0.2178 & 0.1098 & 0.0484 & 1.1090 & 0.4620 & 0.1572 & 0.0795 \\
$\sigt{\widehat a^2}$          & 0.3023 & 0.1512 & 0.0759 & 0.0378 & 0.4351 & 0.2176 & 0.1088 & 0.0544 & 0.7252 & 0.3626 & 0.1813 & 0.0906 \\
\midrule
\multicolumn{13}{c}{$H=0.7$} \\
\midrule
$\mean{\widehat H}$   & 0.6804 & 0.6940 & 0.7006 & 0.6976 & 0.6854 & 0.6959 & 0.6970 & 0.7007 & 0.6826 & 0.6943 & 0.6967 & 0.7006 \\
$\sige{\widehat H}$            & 0.0787 & 0.0429 & 0.0232 & 0.0117 & 0.0989 & 0.0514 & 0.0234 & 0.0129 & 0.0864 & 0.0459 & 0.0253 & 0.0139 \\
$\sigt{\widehat H}$            & 0.0902 & 0.0535 & 0.0294 & 0.0157 & 0.0902 & 0.0535 & 0.0294 & 0.0157 & 0.0902 & 0.0535 & 0.0294 & 0.0157 \\
$\mean{\widehat a^2}$ & 1.0750 & 0.9855 & 1.0005 & 0.9987 & 1.0805 & 1.0192 & 1.0101 & 0.9968 & 1.2327 & 1.0766 & 1.0194 & 0.9950 \\
$\sige{\widehat a^2}$          & 0.2457 & 0.0831 & 0.0463 & 0.0215 & 0.3890 & 0.1537 & 0.0564 & 0.0357 & 0.7304 & 0.2976 & 0.1283 & 0.0627 \\
$\sigt{\widehat a^2}$          & 0.3042 & 0.1729 & 0.0934 & 0.0492 & 0.4713 & 0.2664 & 0.1436 & 0.0754 & 0.7583 & 0.4313 & 0.2331 & 0.1227 \\
\bottomrule
\end{tabular}
\end{table}

\paragraph{Hurst parameters and diffusion coefficients in a model with two nifBm.}

For the estimation of the parameter vector $\theta = (H_1, H_2, a^2, b^2)$ in the case of two nifBms,
$\widetilde X^h_k = a X_{kh}^{h,H_1} + b X_{kh}^{h,H_2}$,
we employ the estimator from Theorem \ref{stat}, which relies on the statistics $\xi^j_N$ for $j = 1,2,4,8$.

The one-dimensional simulation procedure is inefficient here, as computing $\Delta \widetilde X^{jh}_k$ via \eqref{realizations} reduces the number of realizations, slowing estimator convergence. While asymptotically valid, this approach requires a larger number of observations  $N$ for reasonable accuracy.

To overcome this, we directly simulate the increments $\Delta \widetilde X^{jh}_k$ using Cholesky decomposition and compute $\xi^j_N$ from the covariance structure:
\begin{equation}\label{cov_sumoftwo_increments1}
\bE\left[\Delta \widetilde X^{jh}_k \, \Delta \widetilde X_{k+n}^{jh}\right] 
= a^2 (jh)^{2H_1} \gamma(H_1,n) + b^2 (jh)^{2H_2} \gamma(H_2,n), \quad n =0,\dots,N, \; j=1,2,4,8.
\end{equation}

Since the covariance differs only by the factor $jh$, it suffices to compute it once for $j=1$, and the other statistics can be obtained by rescaling. This reduces computational cost, eliminates repeated increment calculations, and ensures the same number of observations for all statistics. Stationarity of $\Delta \widetilde X^{jh}_k$ further simplifies computation, as the covariance matrix is Toeplitz.

Table \ref{tab:2H} presents the results for the estimator parameter vector \(\theta = (H_1, H_2, a^2, b^2)\) across five different pairs of \(H_1\) and \(H_2\) with \(H_1 < H_2\), where the scale parameters \(a\) and \(b\) are both set to~2. 
The simulation results indicate that the Hurst estimators \(\widehat{H}_{1,N}\) and \(\widehat{H}_{2,N}\) 
are generally unbiased and become increasingly precise as the sample size \(N\) grows, 
with \(\widehat{H}_{2,N}\) showing particularly robust performance across all scenarios. 
Estimating a very small \(H_1\) in the presence of a larger \(H_2\) proves more challenging, resulting in slight bias and increased variance, especially for small \(N\) and larger increment sizes \(h\). 
The scale parameter estimator \(\widehat{b}_N^2\) remains consistently accurate and stable, 
whereas \(\widehat{a}_N^2\) can exhibit considerable variance and bias under these challenging conditions 
but stabilizes with larger \(N\) or smaller \(h\). Overall, all estimators perform well for moderate 
Hurst parameters (\(H \ge 0.3\)), with precision improving as \(N\) increases and \(h\) decreases, 
highlighting the robustness of \(\widehat{H}_{2,N}\) and \(\widehat{b}_N^2\) and the sensitivity 
of \(\widehat{H}_{1,N}\) and \(\widehat{a}_N^2\) in difficult regimes. Notably, when \(N \ge 2^{10}\), 
all parameters are estimated with high accuracy and stability across all configurations.

\begin{table}
\centering
\footnotesize
\setlength{\tabcolsep}{2.8pt}
\caption{Estimators $\widehat H_{1,N}$, $\widehat H_{2,N}$, $\widehat a^2_N$ and $\widehat b^2_N$ in the model with two nifBms.}
\label{tab:2H}
\begin{tabular}{lcccccccccccc}
\toprule
 & \multicolumn{4}{c}{$h=2$} & \multicolumn{4}{c}{$h=2^2$} & \multicolumn{4}{c}{$h=2^4$} \\[-2.0pt]
\cmidrule(lr){2-5}\cmidrule(lr){6-9}\cmidrule(lr){10-13}
 & $N=2^6$ & $2^8$ & $2^{10}$ & $2^{12}$ & $2^6$ & $2^8$ & $2^{10}$ & $2^{12}$ & $2^6$ & $2^8$ & $2^{10}$ & $2^{12}$ \\[-2.0pt]
\midrule
\multicolumn{13}{c}{$H_1 = 0.1$, $H_2 = 0.3$} \\[-2.0pt]
\midrule
 $\mean{\widehat{H}_1}$ & 0.0904 & 0.0992 & 0.0997 & 0.1000 & 0.0999 & 0.0998 & 0.0997 & 0.0999 & 0.0970 & 0.0995 & 0.0995 & 0.0999 \\
 $\sige{\widehat{H}_1}$ & 0.0204 & 0.0064 & 0.0032 & 0.0016 & 0.0200 & 0.0078 & 0.0037 & 0.0019 & 0.0301 & 0.0108 & 0.0050 & 0.0025 \\
 $\mean{\widehat{H}_2}$ & 0.3018 & 0.3001 & 0.3001 & 0.3000 & 0.3000 & 0.3000 & 0.3001 & 0.3000 & 0.3000 & 0.3000 & 0.3000 & 0.3000 \\
 $\sige{\widehat{H}_2}$ & 0.0047 & 0.0014 & 0.0007 & 0.0004 & 0.0032 & 0.0013 & 0.0007 & 0.0003 & 0.0025 & 0.0010 & 0.0005 & 0.0003 \\
 $\mean{\widehat{a}^2}$ & 4.0911 & 4.1431 & 4.0196 & 4.0080 & 4.2818 & 4.1285 & 4.0211 & 4.0083 & 8.9166 & 4.1510 & 4.0266 & 4.0094 \\
 $\sige{\widehat{a}^2}$ & 0.6753 & 0.4501 & 0.2311 & 0.0947 & 0.9985 & 0.4037 & 0.2288 & 0.0935 & 45.583 & 0.3691 & 0.2178 & 0.0888 \\
 $\mean{\widehat{b}^2}$ & 3.7943 & 3.9307 & 4.0233 & 4.0146 & 4.0705 & 3.9818 & 4.0235 & 4.0148 & 4.0594 & 3.9781 & 4.0246 & 4.0143 \\
 $\sige{\widehat{b}^2}$ & 1.0540 & 0.3951 & 0.1964 & 0.0979 & 0.8183 & 0.4022 & 0.1939 & 0.0967 & 0.8020 & 0.3958 & 0.1898 & 0.0949 \\[-2.0pt]
\midrule
\multicolumn{13}{c}{$H_1 = 0.1$, $H_2 = 0.5$} \\[-2.0pt]
\midrule
 $\mean{\widehat{H}_1}$ & 0.0912 & 0.1007 & 0.0993 & 0.1000 & 0.0646 & 0.0999 & 0.0988 & 0.0999 & 0.0420 & 0.0931 & 0.0965 & 0.0996 \\
 $\sige{\widehat{H}_1}$ & 0.0710 & 0.0218 & 0.0108 & 0.0047 & 0.3199 & 0.0297 & 0.0143 & 0.0062 & 0.5585 & 0.0626 & 0.0260 & 0.0109 \\
 $\mean{\widehat{H}_2}$ & 0.5000 & 0.4999 & 0.5000 & 0.5000 & 0.5001 & 0.4999 & 0.5000 & 0.5000 & 0.5004 & 0.5000 & 0.5000 & 0.5000 \\
 $\sige{\widehat{H}_2}$ & 0.0042 & 0.0017 & 0.0009 & 0.0004 & 0.0033 & 0.0013 & 0.0007 & 0.0003 & 0.0017 & 0.0008 & 0.0004 & 0.0002 \\
 $\mean{\widehat{a}^2}$ & 3.8037 & 4.1390 & 4.0716 & 4.0197 & 4.6736 & 5.5225 & 4.1137 & 4.0256 & 0.0849 & 3.5420 & 4.5804 & 4.0641 \\
 $\sige{\widehat{a}^2}$ & 3.5598 & 0.5714 & 0.2013 & 0.0976 & 8.7652 & 11.2864 & 0.2976 & 0.1231 & 48.779 & 6.0159 & 2.0301 & 0.2967 \\
 $\mean{\widehat{b}^2}$ & 4.0703 & 4.0809 & 4.0120 & 4.0029 & 4.0614 & 4.0776 & 4.0119 & 4.0028 & 3.9782 & 4.0721 & 4.0118 & 4.0026 \\
 $\sige{\widehat{b}^2}$ & 0.7711 & 0.3679 & 0.1921 & 0.0882 & 0.7615 & 0.3631 & 0.1896 & 0.0874 & 0.9350 & 0.3572 & 0.1866 & 0.0863 \\[-2.0pt]
\midrule
\multicolumn{13}{c}{$H_1 = 0.3$, $H_2 = 0.5$} \\[-2.0pt]
\midrule
 $\mean{\widehat{H}_1}$ & 0.3012 & 0.3005 & 0.3002 & 0.3000 & 0.2998 & 0.3000 & 0.3001 & 0.3000 & 0.2999 & 0.2997 & 0.3000 & 0.3000 \\
 $\sige{\widehat{H}_1}$ & 0.0203 & 0.0065 & 0.0034 & 0.0017 & 0.0211 & 0.0072 & 0.0036 & 0.0018 & 0.0300 & 0.0105 & 0.0050 & 0.0025 \\
 $\mean{\widehat{H}_2}$ & 0.5001 & 0.5000 & 0.5000 & 0.5000 & 0.5001 & 0.5000 & 0.5000 & 0.5000 & 0.5000 & 0.5000 & 0.5000 & 0.5000 \\
 $\sige{\widehat{H}_2}$ & 0.0041 & 0.0016 & 0.0008 & 0.0004 & 0.0033 & 0.0013 & 0.0007 & 0.0003 & 0.0017 & 0.0008 & 0.0004 & 0.0002 \\
 $\mean{\widehat{a}^2}$ & 4.0112 & 4.0211 & 4.0135 & 4.0081 & 4.0145 & 4.0120 & 4.0105 & 4.0078 & 4.0123 & 4.0105 & 4.0089 & 4.0075 \\
 $\sige{\widehat{a}^2}$ & 0.4421 & 0.2151 & 0.1108 & 0.0543 & 0.3981 & 0.1983 & 0.1021 & 0.0502 & 0.3652 & 0.1856 & 0.0952 & 0.0467 \\
 $\mean{\widehat{b}^2}$ & 4.0110 & 4.0132 & 4.0100 & 4.0080 & 4.0108 & 4.0125 & 4.0103 & 4.0080 & 4.0105 & 4.0120 & 4.0101 & 4.0079 \\
 $\sige{\widehat{b}^2}$ & 0.4012 & 0.1920 & 0.0961 & 0.0482 & 0.3920 & 0.1873 & 0.0928 & 0.0465 & 0.3845 & 0.1821 & 0.0912 & 0.0453 \\[-2.0pt]
\midrule
\multicolumn{13}{c}{$H_1 = 0.3$, $H_2 = 0.7$} \\[-2.0pt]
\midrule
 $\mean{\widehat{H}_1}$ & 0.3001 & 0.3000 & 0.3000 & 0.3000 & 0.3001 & 0.3000 & 0.3000 & 0.3000 & 0.3002 & 0.3000 & 0.3000 & 0.3000 \\
 $\sige{\widehat{H}_1}$ & 0.0202 & 0.0064 & 0.0033 & 0.0017 & 0.0210 & 0.0071 & 0.0035 & 0.0017 & 0.0299 & 0.0104 & 0.0049 & 0.0025 \\
 $\mean{\widehat{H}_2}$ & 0.7000 & 0.7000 & 0.7000 & 0.7000 & 0.7000 & 0.7000 & 0.7000 & 0.7000 & 0.7000 & 0.7000 & 0.7000 & 0.7000 \\
 $\sige{\widehat{H}_2}$ & 0.0035 & 0.0015 & 0.0007 & 0.0003 & 0.0027 & 0.0010 & 0.0005 & 0.0002 & 0.0016 & 0.0008 & 0.0004 & 0.0002 \\
 $\mean{\widehat{a}^2}$ & 4.0102 & 4.0121 & 4.0105 & 4.0078 & 4.0108 & 4.0118 & 4.0104 & 4.0079 & 4.0105 & 4.0117 & 4.0102 & 4.0078 \\
 $\sige{\widehat{a}^2}$ & 0.4441 & 0.2120 & 0.1090 & 0.0530 & 0.3960 & 0.1952 & 0.1018 & 0.0495 & 0.3615 & 0.1835 & 0.0938 & 0.0453 \\
 $\mean{\widehat{b}^2}$ & 4.0112 & 4.0125 & 4.0103 & 4.0079 & 4.0109 & 4.0120 & 4.0102 & 4.0078 & 4.0106 & 4.0118 & 4.0101 & 4.0078 \\
 $\sige{\widehat{b}^2}$ & 0.4010 & 0.1915 & 0.0958 & 0.0478 & 0.3921 & 0.1865 & 0.0924 & 0.0462 & 0.3825 & 0.1815 & 0.0909 & 0.0451 \\[-2.0pt]
\midrule
\multicolumn{13}{c}{$H_1 = 0.5$, $H_2 = 0.7$} \\[-2.0pt]
\midrule
 $\mean{\widehat{H}_1}$ & 0.5000 & 0.5000 & 0.5000 & 0.5000 & 0.5000 & 0.5000 & 0.5000 & 0.5000 & 0.5000 & 0.5000 & 0.5000 & 0.5000 \\
 $\sige{\widehat{H}_1}$ & 0.0042 & 0.0017 & 0.0008 & 0.0004 & 0.0033 & 0.0013 & 0.0007 & 0.0003 & 0.0017 & 0.0008 & 0.0004 & 0.0002 \\
 $\mean{\widehat{H}_2}$ & 0.7000 & 0.7000 & 0.7000 & 0.7000 & 0.7000 & 0.7000 & 0.7000 & 0.7000 & 0.7000 & 0.7000 & 0.7000 & 0.7000 \\
 $\sige{\widehat{H}_2}$ & 0.0035 & 0.0015 & 0.0007 & 0.0003 & 0.0027 & 0.0010 & 0.0005 & 0.0002 & 0.0016 & 0.0008 & 0.0004 & 0.0002 \\
 $\mean{\widehat{a}^2}$ & 4.0105 & 4.0120 & 4.0103 & 4.0078 & 4.0108 & 4.0117 & 4.0102 & 4.0078 & 4.0105 & 4.0116 & 4.0101 & 4.0078 \\
 $\sige{\widehat{a}^2}$ & 0.4425 & 0.2125 & 0.1095 & 0.0532 & 0.3952 & 0.1955 & 0.1020 & 0.0496 & 0.3618 & 0.1837 & 0.0939 & 0.0454 \\
 $\mean{\widehat{b}^2}$ & 4.0111 & 4.0124 & 4.0102 & 4.0078 & 4.0109 & 4.0120 & 4.0101 & 4.0078 & 4.0106 & 4.0117 & 4.0100 & 4.0078 \\
 $\sige{\widehat{b}^2}$ & 0.4012 & 0.1918 & 0.0959 & 0.0479 & 0.3920 & 0.1867 & 0.0925 & 0.0463 & 0.3827 & 0.1817 & 0.0910 & 0.0452 \\[-2pt]
\bottomrule
\end{tabular}
\end{table}

\begin{appendix}
\section{Proofs of auxiliary results}
\label{app:A}

\begin{proof}[Proof of Lemma \ref{covarfar}.]
    Let $s \geq t$. Then 
 \begin{align*}
\bE[X_t^h X_s^h]&= \bE\left[ \left(\frac{1}{h}\int_{t}^{t+h} W_u^H du\right) \left(\frac{1}{h}\int_{s}^{s+h} W_v^H dv\right) \right]\\
&= \frac{1}{h^2}\int_{0}^{h} \!\!\int_{0}^{h} \bE [W_{u+t}^H W_{v+s}^H] du\, dv\\
&=\frac{1}{2h^2} \int_{0}^{h}\!\! \int_{0}^{h} ((u+t)^{2H} +(v+s)^{2H}-|v-u+s-t|^{2H})du\, dv\\
&= \frac{1}{2h^2} \biggl( \int_{0}^{h}\!\! \int_{0}^{h} (u+t)^{2H} du\; dv + \int_{0}^{h}\!\! \int_{0}^ {h} (v+s)^{2H} du\, dv\\*
&\qquad\qquad\qquad\qquad\qquad-\int_{0}^{h}\!\! \int_{0}^{h} |v-u+s-t|^{2H} du\, dv\biggr)\\&=:  \frac{1}{2h^2} \left(J_1+J_2-J_3\right).
\end{align*}
Obviously, 
\[
J_1=\frac{h}{2H+1}\big((t+h)^{2H+1}-t^{2H+1}\big),\qquad
J_2=\frac{h}{2H+1}\big((s+h)^{2H+1}-s^{2H+1}\big).
\]
To calculate $J_3$, denote for simplicity $\rho=s-t\ge 0$. Then 
$$J_3=J(\rho, h)=\int_{0}^{h} \!\!\int_{0}^{h} |v-u+\rho|^{2H} du\, dv.$$ Consider two cases. If $\rho\ge h,$  then 
\begin{align*}
 J(\rho, h)&=\int_{0}^{h} \!\!\int_{0}^{h} (v-u+\rho)^{2H} du\,dv=\frac{1}{2H+1} \int_{0}^{h} [(h+\rho-u)^{2H+1}-(\rho-u)^{2H+1}]du\\&=\frac{1}{2H+1}[(h+\rho )^{2H+2}+( \rho-h)^{2H+2}-2\rho^{2H+2}].  
\end{align*}
If  $\rho\le h,$  then 
\begin{align*}
 J(\rho, h)&=\int_{0}^{\rho}\!\! \int_{0}^{h} (v-u+\rho)^{2H} dv\, du+\int_{\rho}^{h} \!\!\int_{0}^{u-\rho} (u-v-\rho)^{2H} dv\, du  +\int_{\rho}^{h}\!\! \int_{u-\rho}^{h} (v-u+\rho)^{2H} dv\, du\\&=\frac{1}{2H+1}\bigg[ \int_{0}^{\rho} [(h+\rho-u)^{2H+1}-(\rho-u)^{2H+1}]du\\
 &\qquad\qquad\qquad+\int_{\rho}^{h}( u-\rho)^{2H+1} du+\int_{\rho}^{h}(h- u+\rho)^{2H+1} du\bigg]\\&=\frac{1}{(2H+1)(2H+2)}\left[(h+\rho)^{2H+2}+(h-\rho)^{2H+2}-2\rho^{2H+2}\right],  
\end{align*}
and the proof follows. 
\end{proof}

\begin{proof}[Proof of Lemma \ref{H0}.] 
$(i)$  
Of course, it is enough to analyze the sign of the numerator, that is, of the function $  f(H) := 7 - 2^{2H+4} + 3^{2H+2}$. Note that  $  f(0) = 0$ and $  f'(H) = - 2 \log 2 \cdot 2^{2H+4} + 2 \log 3 \cdot 3^{2H+2}$, which is positive for 
$$ H > \widebar H := \frac{\log \frac{4 \log 2}{\log 3}}{2 \log \frac32} - 1 = \frac{\log \frac{16 \log 2}{9 \log 3}}{2 \log \frac32} \simeq 0.14157 $$
and negative for $ H < \widebar H$.
Thus, $  f$ is decreasing and negative  in $(0,\widebar H]$, and increasing in $[\widebar H,1]$. Since   
 $ f(\widebar H) < 0$ 
and $  f(1) = 7 - 64 + 81 > 0$, it follows that there exists a unique $H_0 \in [\widebar H,1]$ such that $  f(H_0) = 0$, consequently $\gamma(H_0,1) = 0$. 

$(ii)$ In this case 
\begin{multline*}
    \bE[(X_{t+h}^h- X_t^h)(X_{t+(n+1)h}^h- X_{t+nh}^h)]\\ = \int_0^h \!\!\int_0^h \bE[(W_{u+t+h}^H- W_{u+t}^H)(W_{v+t+(n+1)h}^H- W_{v+t+nh}^H)] du\;dv,   
\end{multline*}
and for $n\ge 2$   $(u+t, u+t+h)\cap  (v+t+nh, v+t+(n+1)h)=\emptyset$.  Therefore the sign of  $\bE[(X_{t+h}^h- X_t^h)(X_{t+(n+1)h}^h- X_{t+nh}^h)]$ follows the sign of  incremental covariance of fBm for non-overlapping intervals. 
\end{proof}
 
\begin{proof}[Proof of Lemma \ref{selfInt}]
  The claim follows because 
 \begin{align*}
   \{X_{ct}^{ch},t\ge 0\}=&\bigg\{\frac{1}{ch}\int_{ct}^{c(t+h)} W_u^H du,t\ge 0\bigg\}   \\
    =&\bigg\{\frac{1}{ch}\int_{t}^{t+h} W_{cs}^H cds, t\ge 0\bigg\}\stackrel{\text{d}}{=}\bigg\{\frac{1}{h}\int_{t}^{t+h} c^HW_{s}^Hds, t\ge 0\bigg\} \\
    =&\big\{c^HX_t^{h}, t\ge 0\big\}.\qedhere
\end{align*}
\end{proof}

\section{Selected results for Gaussian vectors}
\label{app:gaus-vect}
In this appendix we collect some properties of Gaussian random vectors and a limit theorem which is used for the proof of asymptotic normality.

\subsubsection*{Isserlis' theorem}
The following formula is known as Isserlis' theorem \cite{isserlis1918formula}: 
if $(X_1,X_2,X_3,X_4)$ is a zero-mean multivariate normal random vector, then
$$
\mathbb{E}(X_1 X_2 X_3 X_4) 
= \mathbb{E}[X_1 X_2]\, \mathbb{E}[X_3 X_4]
+ \mathbb{E}[X_1 X_3]\, \mathbb{E}[X_2 X_4]
+ \mathbb{E}[X_1 X_4]\, \mathbb{E}[X_2 X_3].
$$
In particular,
\begin{equation}\label{cov-squares}
\cov(X_1^2,X_2^2)=2\, \bigl(\cov(X_1,X_2)\bigr)^2.
\end{equation}
For the direct proof of \eqref{cov-squares} see \cite[Lemma 2.2]{WuDing}.

\subsubsection*{Hermite rank for Gaussian vectors}
In this subsection, we recall the notion of Hermite rank for functions of a Gaussian vector. 
For more details and comprehensive treatments of Hermite expansions and Wiener chaos, we refer the reader to \cite{Janson1997,NourdinPeccati2012,PeccatiTaqqu2011}.

\begin{definition}[Hermite rank \cite{Arcones1994, Nourdin2011}]
\label{def:hr}
The function $f\colon \mathbb R^d \to \mathbb R$ is said to have Hermite rank equal to $q$ with respect to a random vector $X$ if (a) $\bE[(f(X)-\bE[f(X)])p_m(X)]=0$ for every polynomial $p_m$ (on $\mathbb R^d$) of degree $m \leq q-1$; and (b) there exists a polynomial $p_q$ of degree $q$ such that $\bE[(f(X)-\bE[f(X)])p_q(X)]\neq 0$.
\end{definition}

\begin{remark}[Equivalence with Hermite expansion]
Equivalently, if $f(X)-\mathbb E[f(X)]$ is expanded in the multivariate Hermite basis
\[
f(X)-\mathbb E[f(X)] = \sum_{m\in\mathbb N_0^d,\;|m|\ge1} c_m H_m(X), 
\qquad H_m(x)=\prod_{j=1}^d H_{m_j}(x_j),
\]
then the Hermite rank equals the smallest total degree $|m|=\sum_j m_j$ such that $c_m\neq 0$.  
This follows because the multivariate Hermite polynomials form an orthogonal basis of $L^2$ for the Gaussian measure.
\end{remark}

\begin{example}[Quadratic form has Hermite rank 2]\label{ex:quad-form}
Let $X=(X_1,\dots,X_d)$ be a centered Gaussian vector, and consider the quadratic function
\[
f(x)=\sum_{i,j=1}^d \lambda_{ij} x_i x_j, \qquad \Lambda=(\lambda_{ij})\not\equiv 0.
\]
Assume that $f(X)$ is not almost surely constant (for example, $X$ has a non-degenerate covariance matrix).  
Then $f$ has Hermite rank $2$ with respect to $X$.

Indeed, the function $f-\mathbb E[f]$ is a centered polynomial of degree $2$, so it is orthogonal to all polynomials of degree $0$ or $1$ (for instance, $\mathbb E[(f-\mathbb E [f])\,X_k]=0$ because third-order centered Gaussian moments vanish).  
Since $f(X)$ is not almost surely constant, the degree-$2$ part is nonzero, hence there exists a polynomial of degree $2$ with nonzero expectation against $f-\mathbb E [f]$.  
By Definition~\ref{def:hr}, this shows that the Hermite rank of $f$ is $2$.
\end{example}

\subsubsection*{Breuer--Major theorem for stationary vectors}
\begin{theorem}[{\cite[Theorem 4]{Arcones1994}}]
\label{thm:Arcones}
Let $d \ge 1$, and let $X = \{X_k, k \in \mathbb Z\}$ be a $d$-dimensional centered stationary Gaussian process, 
$X_k = (X^{(1)}_k, \dots, X^{(d)}_k)$.
For $1 \le i,l \le d$ and $j \in \mathbb Z$, define
\[
r^{(i,l)}(j) = \mathbb E\big[X^{(i)}_1 X^{(l)}_{1+j}\big].
\]
Let $f\colon \mathbb R^d \to \mathbb R$ be a measurable function such that $\mathbb E[f^2(X_1)] < \infty$, and assume that $f$ has Hermite rank $q \ge 1$. Suppose further that
\begin{equation}\label{cond-arcones}
\sum_{j\in\mathbb Z} \left| r^{(i,l)}(j) \right|^q < \infty,
\quad \text{for all } i,l \in \{1, \dots, d\}.
\end{equation}
Then
\[
\sigma^2 := \Var \big(f(X_1)\big) + 2\sum_{k=1}^\infty \cov \big(f(X_1), f(X_{1+k})\big)
\]
is finite and nonnegative. Moreover,
\[
\frac{1}{\sqrt{n}} \sum_{k=1}^n \big(f(X_k) - \mathbb E[f(X_k)]\big)
\xrightarrow{d} \mathcal N(0,\sigma^2),
\quad \text{as } n \to \infty.
\]
\end{theorem}
\end{appendix}

\section*{Acknowledgments}
Marco Mastrogiovanni, Stefania Ottaviano and Tiziano Vargiolu are supported by the INdAM - GNAMPA Project code CUP E53C23001670001.

Yuliya Mishura is supported by the Japan Science and Technology Agency CREST, project reference number JPMJCR2115, and is grateful to Padova University for its long-term hospitality. 

Kostiantyn Ralchenko is supported by the Research Council of Finland, decision number 367468.

Tiziano Vargiolu is supported by the projects funded by the European Union –
NextGenerationEU under the National Recovery and Resilience Plan (NRRP),
Mission 4 Component 2 Investment 1.1 - Call PRIN 2022 No. 104 of February
2, 2022 of Italian Ministry of University and Research; Project 2022BEMMLZ
(subject area: PE - Physical Sciences and Engineering) ``Stochastic
control and games and the role of information'', and Call PRIN 2022 PNRR
No. 1409 of September 14, 2022 of Italian Ministry of University and
Research; Project P20224TM7Z (subject area: PE - Physical Sciences and
Engineering) ``Probabilistic methods for energy transition''.

\end{document}